\documentclass{elsarticle}
\usepackage{amsmath,amssymb,stmaryrd}
\usepackage{breakurl}
\usepackage{hyperref}
\newcommand{\cat}[1]{\ensuremath{\mathbf{#1}}}
\newcommand{\catC}{\cat{C}}

\newcommand{\op}{\ensuremath{^{\mathrm{op}}}}
\newcommand{\id}[1][]{\ensuremath{1_{#1}}}
\newcommand{\restrict}[1]{\ensuremath{|_{#1}}}
\newcommand{\downset}{\ensuremath{\mathop{\downarrow}}}
\newcommand{\psh}[1]{\ensuremath{[\cat{#1}\op,\cat{Set}]}}
\newcommand{\Psh}[1]{\ensuremath{\widehat{#1}}}
\newcommand{\Spsh}[1]{\ensuremath{\Psh{#1}_{\mathrm{brd}}}}
\newcommand{\FUsh}[1]{\ensuremath{\Psh{#1}_{\mathrm{fin}}}}
\newcommand{\DUsh}[1]{\ensuremath{\Psh{#1}_{\mathrm{dir}}}}
\newcommand{\broad}[2]{\ensuremath{\langle #1, #2\rangle }}
\newcommand{\ootimes}{\ensuremath{\mathop{\widehat{\otimes}}}}
\newcommand{\inprod}[2]{\langle #1 \mid #2 \rangle}
\newcommand{\HModc}{\ensuremath{\cat{Hilb}_{C_0(X)}}} 
\newcommand{\cQuant}{\ensuremath{\cat{cQuant}}}
\newcommand{\mors}[1]{\ensuremath{{#1}^{\preceq}}} 
\newcommand{\suppinit}{\supp_0} 
\newcommand{\qunit}{e} 
\DeclareMathOperator{\ISub}{ISub}
\DeclareMathOperator{\Sub}{Sub}
\DeclareMathOperator{\Simple}{Simple}
\DeclareMathOperator{\supp}{supp}
\DeclareMathOperator{\colim}{colim}

\usepackage{tikz}
\usetikzlibrary{arrows,cd}
\newenvironment{pic}[1][]
{\begin{aligned}\begin{tikzpicture}[#1]}
{\end{tikzpicture}\end{aligned}}

\newcommand{\appref}{\ref{sec:day}}
\journal{Journal of Pure and Applied Mathematics}
\newcommand{\acknowledgements}[1]{}
\usepackage{amsthm}
\theoremstyle{plain}
\newtheorem{theorem}{Theorem}[section]
\newtheorem{proposition}[theorem]{Proposition}
\newtheorem{corollary}[theorem]{Corollary}
\newtheorem{lemma}[theorem]{Lemma}

\theoremstyle{definition}
\newtheorem{definition}[theorem]{Definition}
\newtheorem{example}[theorem]{Example}
\newtheorem{remark}[theorem]{Remark}

\begin{document}
\begin{frontmatter}
  \title{Tensor topology}
  \author{Pau Enrique Moliner}
  \ead{pau.enrique.moliner@ed.ac.uk}
  \author{Chris Heunen\fnref{epsrc}}
  \ead{chris.heunen@ed.ac.uk}  
  \author{Sean Tull\fnref{epsrc}}
  \ead{sean.tull@cs.ox.ac.uk}
  \address{University of Edinburgh}
  \fntext[epsrc]{Supported by EPSRC Research Fellowship EP/L002388/2. 
    We thank 
    Richard Garner,
    Marino Gran,
    Joachim Kock,
    Tim van der Linden,
    Bob Par{\'e},
    John Power,
    Pedro Resende,
    Manny Reyes,
    Phil Scott, and
    Isar Stubbe for helpful discussions.
    }
  \begin{abstract}
  A subunit in a monoidal category is a subobject of the monoidal unit for which a canonical morphism is invertible. 
  They correspond to open subsets of a base topological space in categories such as those of sheaves or Hilbert modules. 
  We show that under mild conditions subunits endow any monoidal category with a kind of topological intuition: there are well-behaved notions of restriction, localisation, and support, even though the subunits in general only form a semilattice. We develop universal constructions completing any monoidal category to one whose subunits universally form a lattice, preframe, or frame. 
  \end{abstract}
  \begin{keyword}
	Monoidal category \sep Subunit \sep Idempotent \sep Semilattice \sep Frame
  \MSC[2010]{18D10, 06D22}
  \end{keyword}
\end{frontmatter}

\section{Introduction}

Categorical approaches have been very successful in bringing topological ideas into other areas of mathematics. A major example is the category of sheaves over a topological space, from which the open sets of the space can be reconstructed as subobjects of the terminal object. More generally, in any topos such subobjects form a frame. Frames are lattices with properties capturing the behaviour of the open sets of a space, and form the basis of the constructive theory of pointfree topology~\cite{johnstone:stonespaces}. 

In this article we study this inherent notion of space in categories more general than those with cartesian products.
Specifically, we argue that a semblance of this topological intuition remains in categories with mere tensor products. 
For the special case of toposes this exposes topological aspects in a different direction than considered previously.
The aim of this article is to lay foundations for this (ambitiously titled) `tensor topology'.

Boyarchenko and Drinfeld~\cite{boyarchenkodrinfeld:idempotent,boyarchenkodrinfeld:duality} have already shown how to equate the open sets of a space with certain morphisms in its monoidal category of sheaves of vector spaces. This forms the basis for our approach. We focus on certain subobjects of the tensor unit in a (braided) monoidal category that we call \emph{subunits}, fitting  other treatments of tensor units~\cite{hines:coherenceselfsimilarity,kock:saavedra,fioreleinster:thompsonsgroupf}. 

For subunits to behave well one requires only that monomorphisms and tensor products interact well; we call a category \emph{firm} when it does so for subunits and \emph{stiff} when it does so globally, after~\cite{quillen:nonunitalrings}. In a firm category subunits always form a (meet-)semilattice. They may have further features, such as having joins that interact with the category through universal properties, and in the strongest case form a frame. We axiomatise such {\emph{locale-based}} 
categories. Aside from toposes, major noncartesian examples are categories of Hilbert modules, with subunits indeed given by open subsets of the base space. More generally, we show how to complete any stiff category {to a locale-based one}.

There are at least two further perspectives on this study.
First, it is analogous to tensor triangular geometry~\cite{balmer:tensortriangulargeometry}, a programme with many applications including algebraic geometry, stable homotopy theory, modular representation theory, and symplectic geometry~\cite{balmer:spectrum,balmerfavi:telescope,balmerkrausestevenson:smashing}. Disregarding triangulations and direct sums, we show that the natural home for many arguments is mere monoidal categories~\cite{hogancamp:idempotent}. We will also not require our categories to be cocomplete~\cite{brandenburg}. 

Second, just as Grothendieck toposes may be regarded as a categorification of frames~\cite{street:topos}, our results may be regarded as categorifying the study of central idempotents in a ring. 
Our algebraic examples include categories of firm nondegenerate modules over a firm nonunital commutative ring, or more generally, over a nonunital bialgebra in a braided monoidal category.

\subsubsection*{Structure of article}

We set out the basics of subunits in Section~\ref{sec:subunits}, showing that they form a semilattice in any firm category.
Section~\ref{sec:examples} introduces our main examples: sheaves, Hilbert modules, modules over a ring, and order-theoretic examples including commutative quantales, generalising frames~\cite{resende:groupoidquantales}. 

In Section~\ref{sec:restriction} we introduce the notion of a morphism `restricting to' a subunit, and show how to turn any subunit into a unit of its restricted category. These \emph{restriction} functors together are seen to form a graded monad. We also show that subunits correspond to certain ideal subcategories and to certain comonads; although these give equivalent descriptions, we stick with subunits for the rest of the article to stay as close as possible to the theory of sheaves.
Section~\ref{sec:simplicity} then proves that restriction forms a localisation of our category, and more broadly that one may localise to a category with only trivial subunits. 

Section~\ref{sec:support} introduces the notion of \emph{support} of a morphism, derived from the collection of subunits to which it restricts. This notion seems unrelated to earlier definitions requiring more structure~\cite{joyal:chevalleytarski,kockpitsch:pointfree}.

In Sections~\ref{sec:spatial} and~\ref{sec:universaljoins} we characterise categories, such as toposes and categories of Hilbert modules, whose subunits come with suprema satisfying universal properties and so form a lattice, preframe, or frame; the latter being {locale-based} categories. Finally, Sections~\ref{sec:broad} and~\ref{sec:completion} show how to complete a given monoidal category to one with each kind of universal joins, including a {locale-based} category, in a universal way. This involves passing to certain presheaves, that we will call \emph{broad}, under Day convolution, as detailed in \appref; but this completion is not a sheafification for any Grothendieck topology.

\subsubsection*{Further directions}

This foundation opens various directions for further investigation. The {locale-based} completion of a stiff category is a stepping stone to future structure theorems for monoidal categories in the spirit of Giraud's theorem~\cite[C2.2.8]{johnstone:elephant}. We therefore also leave to future work more precise connections between tensor topology and topos theory, although the reader might find useful the analogy between objects of a monoidal category and (pre)sheaves.

Applications to linear logic and computer science, as proposed in~\cite{enriquemolinerheunentull:space}, remain to be explored, including amending the graphical calculus for monoidal categories~\cite{selinger:graphicallanguages} with spatial information.
It would be interesting to examine what happens to subunits under constructions such as Kleisli categories, Chu spaces, or the Int-construction~\cite{joyalstreetverity:traced}. One could ask how much of the theory carries over to skew monoidal categories~\cite{szlachanyi:skew}, as topology encompasses more than locale theory and one may be interested in `noncommutative' topologies.
Similarly, one could investigate how these notions relate to partiality and restriction categories~\cite{grandis:cohesive}.
Finally, it would be desirable to find global conditions on a category providing its subunits with further properties, such as being a compact frame or Boolean algebra, or with further structure, such as being a metric space.

\acknowledgements{We thank 
    Robin Cockett,
    Richard Garner,
    Marino Gran,
    Joachim Kock,
    Tim van der Linden,
    Bob Par{\'e},
    John Power,
    Pedro Resende,
    Manny Reyes,
    Phil Scott, and
    Isar Stubbe for helpful discussions.
}

\section{Subunits}\label{sec:subunits}

We work with braided monoidal categories~\cite{maclane:categorieswork}, and will sometimes suppress the coherence isomorphisms $\lambda_A \colon I \otimes A \to A$, $\rho_A \colon A \otimes I \to A$, $\alpha_{A,B,C} \colon A \otimes (B \otimes C) \to (A \otimes B) \otimes C$, and $\sigma_{A,B} \colon A \otimes B \to B \otimes A$, and often abbreviate identity morphisms $\id[A] \colon A \to A$ simply by $A$.

Recall that a subobject of an object $A$ is an equivalence class of monomorphisms $s \colon S \rightarrowtail A$, where $s$ and $s'$ are identified if they factor through each other. Whenever we talk about a subobject, we will use a small letter $s$ for a representing monomorphism, and the corresponding capital $S$ for its domain.

\begin{definition}\label{def:subunit}
  A \emph{subunit} in a braided monoidal category $\cat{C}$ is a subobject $s \colon S \rightarrowtail I$ of the tensor unit such that $s \otimes S \colon S \otimes S \to I \otimes S$ is an isomorphism\footnote{
Boyarchenko and Drinfeld call morphisms $s \colon S \to I$ for which $s \otimes S$ and $S \otimes s$ are isomorphisms \emph{open idempotents}~\cite{boyarchenkodrinfeld:idempotent}, with (the dual of) this notion going back implicitly at least to~\cite[Exercise~4.2]{kashiwara2005categories}. In~\cite{enriquemolinerheunentull:space} subunits were called \emph{idempotent subunits}.}.
Write $\ISub(\cat{C})$ for the collection of subunits in $\cat{C}$.
\end{definition}

Note that, because $s$ is monic, if $s \otimes S$ is invertible then so is $S \otimes s$.

\begin{remark}\label{rem:centrality}
  We could have generalised the previous definition to arbitrary monoidal categories by additionally requiring subunits to be central in the sense that there is a natural isomorphism $(-) \otimes S \Rightarrow S \otimes (-)$. 
  Most results below still hold, but the bureaucracy is not worth the added generality here. 

  Many results also remain valid when we require $s \otimes S$ not to be invertible but merely split epic, but for simplicity we stick with invertibility.
\end{remark}

We begin with some useful observations, mostly adapted from Boyarchenko and Drinfeld~\cite{boyarchenkodrinfeld:idempotent}.

\begin{lemma}\label{lem:subunitsretract} 
  Let $m \colon A \to B$ and $e \colon B \to A$ satisfy $e \circ m = A$, and $s \colon S \rightarrowtail I$ be a subunit.
  If $s \otimes B$ is an isomorphism, then so is $s \otimes A$.
\end{lemma}
\begin{proof}
  The diagram below commutes by bifunctoriality of $\otimes$.
  \[\begin{tikzcd}[column sep=3cm]
    S \otimes A \rar{S \otimes m} \dar{s \otimes A}
    & S \otimes B \rar{S \otimes e} \arrow{d}{s \otimes B}[swap]{\simeq}
    & S \otimes A \dar{s \otimes A} \\
    I \otimes A \rar{I \otimes m}
    & I \otimes B \rar{I \otimes e}
    & I \otimes A
  \end{tikzcd}\]
  Both rows compose to the identity, and the middle vertical arrow is an isomorphism.
  Hence $s \otimes A$ is an isomorphism with inverse $(S \otimes e) \circ (s \otimes B)^{-1} \circ (I \otimes m)$.
\end{proof}

Recall that subobjects of a fixed object always form a partially ordered set, where $s \leq t$ if and only if $s$ factors through $t$. The following observations characterises this order in another way for subunits.

\begin{lemma}\label{lem:subunitsorder}
  A subunit $s$ factors through another $t$ if and only if $S \otimes t$ is invertible, or equivalently, $t \otimes S$ is invertible.
\end{lemma}
\begin{proof}
  Suppose $s=t \circ f$.
  Set $g=(S \otimes f) \circ (S \otimes s)^{-1} \circ \rho_S^{-1} \colon S \to S \otimes T$. Then
  \[
    \rho_S \circ (S \otimes t) \circ g
    = \rho_S \circ (S \otimes s) \circ (S \otimes s)^{-1} \circ {\rho_S}^{-1} 
    = S\text.
  \]
  Idempotence of $t$ makes $S \otimes T \otimes t \colon S \otimes T \otimes T \to S \otimes T \otimes I$ an isomorphism. Hence, by the right-handed version of Lemma~\ref{lem:subunitsretract}, so is $S \otimes t$. A symmetric argument makes $t \otimes S$ invertible.

  Conversely, suppose $S \otimes t$ is an isomorphism. Because the diagram
  \[\begin{tikzcd}[column sep=3cm]
    S \otimes T \dar{S \otimes t} \rar{s \otimes T}
    & I \otimes T \dar{I \otimes t} \rar{\rho_T}
    & T \dar{t} \\
    S \otimes I \rar{s \otimes I}
    & I \otimes I \rar{\rho_I}
    & I
  \end{tikzcd}\]
  commutes, the bottom row $s \circ \rho_S$ factors through the right vertical arrow $t$, whence so does $s$.
\end{proof}

It follows from Lemma~\ref{lem:subunitsorder} that subunits are determined by their domain: if $s, s' \colon S \rightarrowtail I$ are subunits, then $s' = s \circ f$ for a unique $f$, which is an isomorphism. This justifies our convention to use the same letter for a subunits and its domain. 

For the theory to work smoothly, we impose a condition on the category. 

\begin{definition}\label{def:firm}
  A category is called \emph{firm} when it is braided monoidal and $s \otimes T \colon S \otimes T \to I \otimes T$ is a monomorphism whenever $s$ and $t$ are subunits.
\end{definition}

\begin{remark}
  The name \emph{firm} is chosen after Quillen~\cite{quillen:nonunitalrings}, who employs it as a natural condition for nonunital rings to make up for a missing unit. The previous definition extends the term to the category of nonunital rings; see {Proposition}~\ref{prop:modules} below. Note, however, that a firm category has genuine identity morphisms and a genuine tensor unit. Firmness is very mild condition: {Proposition}~\ref{prop:PshFirmCounter} below gives a category that is not firm, but we know of no other `naturally occurring' categories that are not firm. 
\end{remark}

\begin{lemma}
  Any co-closed braided monoidal category is firm.
\end{lemma}
\begin{proof}
 Each functor $(-) \otimes T$ is a right adjoint and so preserves limits and hence monomorphisms. Hence whenever $s$ is monic so is $s \otimes T$.
\end{proof}

In particular, a $*$-autonomous category is firm, as is a compact category.

\begin{remark}
  In the following, we will completely disregard size issues, and pretend $\ISub(\cat{C})$ is a set, as in our main examples.
\end{remark}

\begin{proposition}\label{prop:semilattice}
  The subunits in a firm category form a semilattice, with largest element $I$, meets given by
  \[
    \big(s \colon S \rightarrowtail I\big) \wedge \big(t \colon T \rightarrowtail I\big) = \big(\lambda_I \circ (s \otimes t) \colon S \otimes T \rightarrowtail I\big)\text,
  \]
  and the usual order of subobjects.
\end{proposition}
\begin{proof}
  First observe that $s \otimes t = (I \otimes t) \circ (s \otimes T)$ is monic, because $I \otimes t = \lambda_I^{-1} \circ t \circ \lambda_T$ is monic, and $s \otimes T$ is monic by firmness.
  It is easily seen to be idempotent using the braiding, and hence it is a well-defined subunit.

  Next, we show that $\ISub(\cat{C})$ is an idempotent commutative monoid under $\wedge$ and $I$.
  The subunit $I$ is a unit as $I \otimes s = \lambda_I \circ (I \otimes s) = s \circ \lambda_S$  represents the same subobject as $s$, and similarly $I \otimes s$ represents the same subobject as $s$ because $\rho_I=\lambda_I$.
  An analogous argument using coherence establishes associativity. For commutativity, use the braiding to observe that $s \otimes t$ and $t \otimes s$ represent the same subobject. For idempotence note that $s \otimes s$ and $s$ represent the same subobject because $\lambda_I \circ (s \otimes s) = s \circ \rho_S \circ (S \otimes s)$.
  
  Hence $\ISub(\cat{C})$ is a semilattice where $s$ is below $t$ if and only if $s = s \wedge t$. Finally, we show that this order is the same as the usual order of subobjects. 
  On the one hand, if $s$ and $s \otimes t$ represent the same subobject, then $S \simeq S \otimes T$, making $S \otimes t$ an isomorphism and so $s \leq t$ by Lemma~\ref{lem:subunitsorder}.
    \[
  \begin{pic}[xscale=2,yscale=1]
    \node (s) at (0,0) {$S$};
    \node (i) at (1,.5) {$I$};
    \node (t) at (0,1) {$T$};
    \draw[>->] (s) to node[below]{$s$} (i);
    \draw[>->] (t) to node[above]{$t$} (i);
    \draw[>->,dashed] (s) to (t);
  \end{pic}
  \qquad \iff \qquad
  \begin{pic}[xscale=2,yscale=1]
    \node (s) at (0,0) {$S$};
    \node (i) at (1,0) {$I$};
    \node (st) at (0,1) {$S \otimes T$};
    \node (ii) at (1,1) {$I \otimes I$};
    \draw[>->] (s) to node[below]{$s$} (i);   
    \draw[>->] (st) to node[above]{$s \otimes t$} (ii);
    \draw[->] (ii) to node[right]{$\lambda_I$} node[left]{$\simeq$} (i);
    \draw[->,dashed] (s) to node[left]{$\simeq$} (st);
  \end{pic}
  \]
  On the other hand, if $s \leq t$ then by the same lemma $S \otimes t$ is an isomorphism with $s = \lambda_I \circ (s \otimes t) \circ (S \otimes t)^{-1} \otimes \rho_S^{-1}$, and so both subobjects are equal. 
\end{proof}

\section{Examples}\label{sec:examples}

This section determines the subunits of four families of examples: 
cartesian categories, like sheaves over a topological space;
commutative unital quantales;
firm modules over a nonunital ring;
and Hilbert modules over a nonunital commutative C*-algebra.

\subsection*{Cartesian categories}

We start with examples in which the tensor product is in fact a product.

{
\begin{proposition}\label{prop:cartesian}
  Any cartesian category $\cat{C}$ is firm, and $\ISub(\cat{C})$ consists of the subobjects of the terminal object.

  In particular, if $X$ is a topological space, then subunits in its category of sheaves $\mathrm{Sh}(X)$ correspond to open subsets of $X$~\cite[Corollary~2.2.16]{borceux:3}.
\end{proposition}
}
\begin{proof}
  Let $s \colon S \rightarrowtail 1$ be a subterminal object.
  Let $\Delta =  \langle S, S \rangle \colon S \to S \times S$ be the diagonal and write $\pi_i \colon A_1 \times A_2 \to A_i$ for the projections.
  Then $(s \times S) \circ \Delta \circ \pi_2 = \pi_2^{-1} \circ S \circ \pi_2 = 1 \times S$. Now, the unique map $s$ of type $S \to 1$ is monic precisely when any two parallel morphisms into $S$ are equal.
  Hence $\pi_i \circ \Delta \circ \pi_2 \circ (s \times S) = \pi_i$, and so $\Delta \circ \pi_2 \circ (s \times S) = \langle \pi_1,\pi_2 \rangle = S \times S$.
  Thus $s \times S$ is automatically invertible.

  Finally, suppose $s_i \colon S_i \rightarrowtail 1$ for $i=1,2$ are monic, and that $f,g \colon A \to S_1 \times S_2$ satisfy $(s_1 \times s_2) \circ f = (s_1 \times s_2) \circ g$. Postcomposing with $\pi_i$ shows that $s_i \circ \pi_i \circ f = s_i \circ \pi_i \circ g$, whence $\pi_i \circ f = \pi_i \circ g$ and so $f=g$. This establishes firmness.
\end{proof}

\subsection*{Semilattices}

Next we consider examples that are degenerate in another sense: firm categories in which there is at most one morphism between two given objects.

{
\begin{example}\label{ex:semilattice}
  Any semilattice $(L, \wedge, 1)$ forms a strict symmetric monoidal category: objects are $x \in L$, there is a unique morphism $x \to y$ if $x \leq y$, tensor product is given by meet, and tensor unit is $I = 1$. Every morphism is monic so this monoidal category is firm, and its (idempotent) subunits are $(L, \wedge, 1)$.
\end{example}

  This gives the free firm category on a semilattice. More precisely, this construction is left adjoint to the functor from the category $\cat{Firm}$ of firm categories with (strong) monoidal subunit-preserving functors to the category $\cat{SLat}$ of semilattices and their homomorphisms, which takes subunits.
  \[\begin{pic}
  	\node (l) at (0,0) {$\cat{SLat}$};
  	\node (r) at (4,0) {$\cat{Firm}$};
  	\draw[->] ([yshift=2mm]l.east) to ([yshift=2mm]r.west);
  	\draw[draw=none] (l.east) to node{$\perp$} (r.west);
  	\draw[<-] ([yshift=-2mm]l.east) to node[below]{$\ISub$} ([yshift=-2mm]r.west);
  \end{pic}\]
}

\subsection*{Quantales}

We move on to more interesting examples, namely special kinds of semilattices like frames and quantales.

\begin{definition}\label{def:frame}
  A \emph{frame} is a complete lattice in which finite joins distribute over suprema.
  A morphism of frames is a function that preserves $\bigvee$, $\wedge$, and $1$.
  Frames and their morphisms form a category $\cat{Frame}$.
\end{definition}

The prototypical example of a frame is the collection of open sets of a topological space~\cite{johnstone:stonespaces}. Frames may be generalised as follows~\cite{rosenthal:quantales}.

\begin{definition}\label{def:quantale}
  A \emph{quantale} is a monoid in the category of complete lattices. 
  More precisely, it is a partially ordered set $Q$ that has all suprema, that has a multiplication $Q \times Q \to Q$, and that has an element $\qunit$, such that:
  \[
    a \big(\bigvee b_i\big) = \bigvee a b_i, \qquad
    \big(\bigvee a_i\big) b = \bigvee a_i b, \qquad
    a \qunit = a = \qunit a.
  \]
  A morphism of quantales is a function that preserves $\bigvee$, $\cdot$, and $\qunit$.
  A quantale is \emph{commutative} when $ab = ba$ for all $a, b \in Q$. Commutative  quantales and their morphisms {form a monoidal category} $\cQuant$.
\end{definition}

Equivalently, a frame is a commutative quantale in which the multiplication is idempotent {and whose unit is the largest element}. 

Any quantale may be regarded as a monoidal category, whose objects are elements of the quantale, where the (composition of) morphisms is induced by the partial order, and the tensor product is induced by the multiplication. This monoidal category is firm, but only braided if the quantale is commutative. 

{
\begin{proposition}\label{ex:quantale}
  Taking subunits is right adjoint to the inclusion:
  \[\begin{pic}
  	\node (l) at (0,0) {$\cat{Frame}$};
  	\node (r) at (4,0) {$\cQuant$};
  	\draw[->] ([yshift=2mm]l.east) to ([yshift=2mm]r.west);
  	\draw[draw=none] (l.east) to node{$\perp$} (r.west);
  	\draw[<-] ([yshift=-2mm]l.east) to node[below]{$\ISub$} ([yshift=-2mm]r.west);
  	\node (L) at (-1,-1) {$\{ q \in Q \mid q^2=q \leq \qunit \}$};
  	\node (R) at (3.4,-1) {$Q$};
  	\draw[|->] (R.west) to (L.east);
  \end{pic}\]
\end{proposition}
}
\begin{proof}
  We first prove that $\ISub(Q)$ is a well-defined frame.
  If $q_i \in \ISub(Q)$, 
  \[
    (\bigvee q_i)^2 = \bigvee_{i,j} q_i q_j \leq \bigvee_{i,j} q_i \qunit = \bigvee_i q_i
    = \bigvee_i q_i q_i \leq \bigvee_{i,j} q_i q_j =(\bigvee q_i)^2
  \]
  and $\bigvee q_i \leq \bigvee_i \qunit = \qunit$, so $\bigvee q_i \in \ISub(Q)$.
  Moreover, if $p,q \in \ISub(Q)$, then $pq$ is again below $\qunit$ and is idempotent by commutativity of $Q$. 
  {It follows that $pq=p \wedge q$ in $\ISub(Q)$.}
  Since quantale multiplication distributes over suprema, so do finite meets.
  
  For the adjunction, observe that if $F$ is a frame and $Q$ is a commutative quantale, then $F= \ISub(F)$ and any morphism $F \to Q$ of quantales restricts to a unique morphism of frames $F \to \ISub(Q)$. 
\end{proof}

\begin{remark}
  {Propositions}~\ref{prop:cartesian} and~\ref{ex:quantale} show that subunits do not capture all possible topological {content} in the traditional sense. For a Grothendieck topos they form the poset of internal truth values, which does not suffice to reconstruct the category, which may itself be said to embody a notion of topological space. 

  For the commutative quantale, $[0,\infty]$ under multiplication and the usual order, the subunits form the two-element Boolean algebra, which is clearly far poorer than the quantale itself.
\end{remark}

\begin{example}
  If $M$ is a monoid, then its (right) ideals form a unital quantale $Q$ with multiplication $IJ=\{xy \mid x \in I,y \in J\}$ and unit $M$ itself. When $M$ is commutative, so is $Q$, and $\ISub(Q)$ consists of all ideals satisfying $I=II$. 
\end{example}

{
\begin{example}
  If $R$ is a commutative ring, then its additive subgroups form a unital commutative quantale $Q$ with multiplication $GH=\{x_1y_1+\cdots+x_ny_n\mid x_i \in G, y_i \in H\}$, supremum $\bigvee G_i = \{\sum_{j \in J} x_j \mid x_j \in G_j \text{ for } J \subseteq I \text{ finite}\}$, and unit $\mathbb{Z}1 = \{ 0, 1, -1, 1+1, -1-1, 1+1+1, -1-1-1, \ldots\}$.
  Then $G \leq H$ iff $G \subseteq H$ and $\ISub(Q)$ consists of those subgroups $G$ such that $G \subseteq G \cdot G$ and $G \subseteq \mathbb{Z}1$. 
  The latter means that $G$ must be of the form $n\mathbb{Z}1$ for some $n \in \mathbb{N}$.
  The former then means that $n1=n^2y1$ for some $y \in \mathbb{Z}$.
  Thus $\ISub(Q) = \{n\mathbb{Z}1 \mid n \in \mathbb{N}, \exists y \in \mathbb{Z} \colon n1=n^2y1\}$.
\end{example}
}

\subsection*{Modules} 

Another example of a monoidal category is that of modules over a ring. We have to take some pains to treat nonunital rings.

{
\begin{definition}\label{def:firmring}
  A commutative ring $R$ is \emph{idempotent} when $R$ equals $R^2=\{\sum_{i=1}^n r_i' r_i'' \mid r_i',r_i'' \in R\}$, \emph{firm} when its multiplication is a bijection $R \otimes_R R \to R$, and \emph{nondegenerate} when $r \in R$ vanishes as soon as $rs=0$ for all $s \in R$.
\end{definition}

Any unital ring is firm and nondegenerate, but examples also include infinite direct sums $\bigoplus_{n \in \mathbb{N}} R_n$ of unital rings $R_n$. Firm rings $R$ are idempotent.

\begin{definition}\label{def:firmmodule}
  Let $R$ be a nondegenerate firm commutative ring. An $R$-module $E$ is \emph{firm} when the scalar multiplication is a bijection $E \otimes_R R \to E$~\cite{quillen:nonunitalrings}, and \emph{nondegenerate} when $x \in E$ vanishes as soon as $xr=0$ for all $r \in R$.
  Nondegenerate firm $R$-modules and linear maps form a monoidal category $\cat{FMod}_R$.
\end{definition}

If $R$ is unital, then every unital $R$-module is firm and nondegenerate.}

{
\begin{proposition}\label{prop:modules}
  The subunits in $\cat{FMod}_R$ correspond to \emph{nondegenerate firm idempotent ideals}: ideals $S \subseteq R$ that are idempotent as rings, and nondegenerate and firm as $R$-modules. 
  Any ideal that is unital as a ring is a nondegenerate firm idempotent ideal. 
  The category $\cat{FMod}_R$ is firm.
\end{proposition}
}
\begin{proof}
  Monomorphisms are injective by nondegeneracy, so every subunit is a nondegenerate firm $R$-submodule of $R$, that is, a nondegenerate firm ideal. Because the inclusion $S \otimes S \to R \otimes S$ is surjective and $S$ is firm, the map $S \otimes S \to S$ given by $s' \otimes s''\mapsto s's''$ is surjective. Thus $S$ is idempotent.
  
  Conversely, let $S$ be a nondegenerate firm idempotent ideal of $R$. The inclusion $S \otimes S \to R \otimes S$ is surjective, as $r \otimes s \in R \otimes S$ can be written as $r \otimes s's'' = r s' \otimes s'' \in S \otimes S$. Hence $S$ is a subunit.
  
  Next suppose ideal $S$ is unital (with generally $1_S \neq 1_R$ if $R$ is unital).
  Then $S \otimes R \to S$ given by $s \otimes r \mapsto sr$ is bijective: surjective as $1_S \otimes s \mapsto 1_S s = s$; and injective as $s \otimes r = 1_S \otimes sr = 1_S \otimes 0 = 0$ if $sr=0$.
  Hence $S$ is firm and nondegenerate.
  Any $s \in S$ can be written as $s=s 1_S \in S^2$, so $S$ is idempotent.

  Finally, to see that the category is firm, let $S,T \subseteq R$ be nondegenerate firm idempotent ideals.
  We need to show that the map $S \otimes T \mapsto R \otimes T$ given by $s \otimes t \mapsto s \otimes t$ is injective.
  Because $T$ is firm, it suffices that multiplication $S \otimes T \to S$ given by $s \otimes t \mapsto st$ is injective, which holds because $S$ is firm.
\end{proof}

The previous proposition generalises to commutative nonunital bialgebras in any symmetric monoidal category.

\begin{example}\label{ex:modulesoverbialgebra}
 Let $\cat{C}$ be a symmetric monoidal category.
  A \emph{commutative nonunital bialgebra} in $\cat{C}$ is an object $M$ together with an associative multiplication $\mu \colon M \otimes M \to M$ and a comonoid $\delta \colon M \to M \otimes M$, $\varepsilon \colon M \to I$, for which $\mu$ and $\delta$ are commutative and satisfy both $ \varepsilon \circ \mu = \varepsilon \otimes \varepsilon$ and the bialgebra law:
\[
    (\mu \otimes \mu) \circ (M \otimes \sigma \otimes M) \circ (\delta \otimes \delta)  = \delta \circ \mu 
\]
We define a braided monoidal category $\cat{Mod}_M$ where objects are $\alpha \colon M \otimes A \to A$ satisfying $\alpha \circ (\mu \otimes A) = \alpha \circ (M \otimes \alpha)$, with morphisms and $\otimes$ all defined as for modules over a (unital) commutative bialgebra (see e.g.~\cite[2.2,2.3]{hasegawa2010bialgebras}). 
The category $\cat{Mod}_M$ is firm when $\cat{C}$ is, and its subunits correspond to \emph{firm ideals}:
  monomorphisms $s \colon S \rightarrowtail M$ such that
  \[\begin{tikzpicture}[xscale=3]
  	\node (tl) at (0,1) {$M \otimes S$};
  	\node (tr) at (1,1) {$M \otimes M$};
  	\node (bl) at (0,0) {$S$};
  	\node (br) at (1,0) {$M$};
  	\draw[->] (tl) to node[above]{$M \otimes s$} (tr);
  	\draw[->] (tr) to node[right]{$\mu$} (br);
  	\draw[>->] (bl) to node[below]{$s$} (br);
  	\draw[->,dashed] (tl) to (bl);
  \end{tikzpicture}\]
  and $\varepsilon \otimes S$ and $s \otimes S$ are isomorphisms.
\end{example}

We next instantiate the previous example in two special cases: in the monoidal categories of semilattices and of quantales.

\begin{example}
  Any semilattice $M$ is a nondegenerate nonunital bialgebra in $\cat{SLat}$. In $\cat{Mod}_M$ objects are semilattices $A$ with functions $\alpha \colon M \times A \to A$ which respect $\wedge$ in each argument and satisfy $\alpha(x \wedge y,a)=\alpha(x,\alpha(y,a))$. 
  Subobjects of the tensor unit correspond to subsets $S \subseteq M$ which are ideals under $\wedge$, or equivalently downward-closed. Because $x \otimes y = (x \wedge x) \otimes y = x \otimes (x \wedge y) \in S \otimes S$, we have $S \otimes S = S \otimes M$, and every subobject of the tensor unit is a subunit.
\end{example}

\begin{example}
  Any commutative unital quantale $M$ is a nondegenerate nonunital bialgebra in the category of complete lattices. $\cat{Mod}_M$ then consists of complete lattices $A$ with functions $\alpha \colon M \times A \to A$ preserving arbitrary suprema in each argument and with $\alpha(x, \alpha(y,a)) = \alpha(xy, a)$.
  Subobjects of the tensor unit are {submodules} $S \subseteq M$. 
  Subunits furthermore have that for every $r \in S$ and $x \in M$ there exist $s_i,t_i \in S$ with $r \otimes x = \bigvee s_i \otimes t_i$.
  For example, if $M=[0,\infty]$ under addition with the opposite ordering, subunits include $\emptyset,\{\infty\}$, $\{0,\infty\}$, $(0,\infty]$, and $[0,\infty]$.
\end{example}

\subsection*{Hilbert modules}

The above examples of module categories were all algebraic in nature. Our next suite of examples is more analytic. 

\begin{definition}\label{def:hilbertmodule}
  Fix a locally compact Hausdorff space $X$. It induces a commutative C*-algebra 
  \[
    C_0(X)=\{f \colon X \to \mathbb{C} \text{ continuous} \mid \forall \varepsilon>0\; \exists K \subseteq X \text{ compact} \colon |f(X \setminus K)|<\varepsilon \}\text.
  \]
  A \emph{Hilbert module} is a $C_0(X)$-module $A$ with a map $\inprod{-}{-} \colon A \times A \to C_0(X)$ that is $C_0(X)$-linear in the second variable, satisfies $\inprod{a}{b}=\inprod{b}{a}^*$, and $\inprod{a}{a}\geq 0$ with equality only if $a=0$, and makes $A$ complete in the norm $\|a\|^2_A = \sup_{x \in X} \inprod{a}{a}(x)$. A function $f \colon A \to B$ between Hilbert $C_0(X)$-modules is \emph{bounded} when $\|f(a)\|_{{A}} \leq {C}\|a\|_A$ for some {constant $C \in \mathbb{R}$; the infimum of such constants is written $\|f\|$}. Here we will focus on {\emph{nonexpansive maps}}, i.e.\ those bounded functions with $\|f\| \leq 1$.
\end{definition}

Hilbert modules were first introduced by Kaplansky~\cite{kaplansky:modules} and studied by many others, including Rieffel~\cite{rieffel:representations} {and} Kasparov~\cite{kasparov:hilbertmodules}. 
For more information we refer to~\cite{lance:hilbertmodules}.

The category $\HModc$ of Hilbert $C_0(X)$-modules and {nonexpansive} $C_0(X)$-linear maps is not abelian, not complete, and not cocomplete~\cite{heunen:embedding}. Nevertheless, $\HModc$ is 
symmetric monoidal~\cite[Proposition~2.2]{heunenreyes:frobenius}. Here $A \otimes B$ is constructed as follows: consider the algebraic tensor product of $C_0(X)$-modules, and complete it to a Hilbert module with inner product $\inprod{a \otimes b}{a' \otimes b'}$ given by $\inprod{a}{a'} \inprod{b}{b'}$. The tensor unit is $C_0(X)$ itself, which forms a Hilbert $C_0(X)$-module under the inner product $\inprod{f}{g}(x) = f(x)^* g(x)$.

{
\begin{proposition}\label{prop:hilbertmodules}
$\HModc$ is firm, and its subunits are 
  \begin{equation}\label{eq:subunithilbertmodules}
    \{ f \in C_0(X) \mid f(X \setminus U)=0 \} \simeq C_0(U)
  \end{equation}
  for open subsets $U \subseteq X$. 
\end{proposition}
}
\begin{proof}
  If $U$ is an open subset of $X$, we may indeed identify $C_0(U)$ with the closed ideal of $C_0(X)$ in~\eqref{eq:subunithilbertmodules}: if $f \in C_0(U)$, then its extension by zero on $X \setminus U$ is in $C_0(X)$, and conversely, if $f \in C_0(X)$ is zero outside $U$, then its restriction to $U$ is in $C_0(U)$.  
  Moreover, note that the canonical map $C_0(X) \otimes C_0(X) \to C_0(X)$ is always an isomorphism as $C_0(X)$ is the tensor unit, and hence the same holds for $C_0(U)$.
  Thus $C_0(U)$ is a subunit in $\HModc$.

  For the converse, let $s \colon S \rightarrowtail C_0(X)$ be a subunit in $\HModc$. We will show that $s(S)$ is a closed ideal in $C_0(X)$, and therefore of the form $C_0(U)$ for some open subset $U \subseteq X$.
  It is an ideal because $s$ is $C_0(X)$-linear.
  To see that it is closed, let $g \in s(S)$. Then
  \begin{align*}
    \|g\|_S^4    
    & = \| \inprod{g}{g}_S^2 \|_{C_0(X)}
      = \| \inprod{g}{g}_S \inprod{g}{g}_S\|_{C_0(X)} \\
    & = \| \inprod{g \otimes g}{g \otimes g}_{C_0(X)} \|_{C_0(X)}
     = \|g \otimes g\|_S^2 \\
    & \leq \|\rho_S^{-1}\|^2 \|g^2\|_S 
    = \|\rho_S^{-1}\|^2 \| \inprod{g}{g}_S g^* g \|_{C_0(X)} \\
    & \leq \|\rho_S^{-1}\|^2 \|g\|_S^2 \|g\|_{C_0(X)}^2
  \end{align*}
  and therefore $\|g\|_S \leq \| \rho_S^{-1}\|^2 \|g\|_{C_0(X)}^2$.
  Because $s$ is bounded, it is thus an equivalence of normed spaces between $(S,\|-\|_S)$ and $(s(S), \|-\|_{C_0(X)})$. Since the former is complete, so is the latter. Firmness follows from {Proposition}~\ref{prop:hilbertmodules:localadjunction} later.
\end{proof}

The category $\HModc$ can be adapted to form a dagger category by considering (not necessarily {nonexpansive}) bounded maps between Hilbert modules to that are \emph{adjointable}. In that case only clopen subsets of $X$ correspond to subunits~\cite[Lemma~3.3]{heunenreyes:frobenius}.

Another way to view a Hilbert $C_0(X)$-module is as a \emph{field of Hilbert spaces} over $X$. Intuitively, this assigns to each $x \in X$ a Hilbert space, that `varies continuously' with $x$. In particular, for each $x \in X$ there is a monoidal functor $\HModc \to \cat{Hilb}_{\mathbb{C}}$. For details, see~\cite{heunenreyes:frobenius}. This perspective may be useful in reading Section~\ref{sec:restriction} later. 

Not every subobject of the tensor unit in $\HModc$ is induced by an open subset $U \subseteq X$, and so the condition of Definition~\ref{def:subunit} is not redundant.

{
\begin{proposition}
  Let $X=[0,1]$. If $f \in C_0(X)$, write $\hat{f} \in C_0(X)$ for the map $x \mapsto xf(x)$.
  Then $S=\{ \hat{f} \mid f \in A\}$ is a subobject of $A=C_0(X)$ in $\HModc$ under $\inprod{\hat{f}}{\hat{g}}_S = \inprod{f}{g}_A$, that is not closed under $\|-\|_A$.
\end{proposition}
}
\begin{proof}
  Clearly $S$ is a $C_0(X)$-module, and $\inprod{-}{-}_S$ is sesquilinear.
  Moreover $S$ is complete: $\hat{f_n}$ is a Cauchy sequence in $S$
  if and only if $f_n$ is a Cauchy sequence in $A$, in which case it converges in $A$ to some $f$, and so $\hat{f_n}$ converges to $\hat{f}$ in $S$.
  Thus $S$ is a well-defined Hilbert module.
  The inclusion $S \hookrightarrow A$ is bounded and injective, and hence a well-defined monomorphism.
  In fact, $A$ is a C*-algebra, and $S$ is an ideal. The closure of $S$ in $A$ is the closed ideal $\{f \in C_0(X) \mid f(0)=0\}$, corresponding to the closed subset $\{0\} \subseteq X$. It contains the function $x \mapsto \sqrt{x}$ while $S$ does not, and so $S$ is not closed. 
\end{proof}

\section{Restriction}\label{sec:restriction}

Regarding subunits as open subsets of an (imagined) base space, the idea of restriction to such an open subset makes sense. For example, if $U$ is an open subset of a locally compact Hausdorff space $X$, then any $C_0(X)$-module induces a $C_0(U)$-module, and any sheaf over $X$ induces a sheaf over $U$. More generally, any subunit in a topos induces an open subtopos.
This section shows that this restriction behaves well in any monoidal category.

\begin{definition}\label{def:supportin}
  A morphism $f \colon A \to B$ \emph{restricts to} a subunit $s \colon S \to I$ when it factors through $\lambda_B \circ (s \otimes B)$. 
  \[\begin{tikzpicture}[xscale=3,yscale=1.25]
    \node (tl) at (0,1) {$A$};
    \node (tr) at (1,1) {$B$};
    \node (bl) at (0,0) {$S \otimes B$};
    \node (br) at (1,0) {$I \otimes B$};
    \draw[->] (tl) to node[above]{$f$} (tr);
    \draw[->,dashed] (tl) to (bl);
    \draw[->] (bl) to node[below]{$s \otimes B$} (br);
    \draw[->] (br) to node[right]{$\lambda_B$} (tr);
  \end{tikzpicture}\]
\end{definition}

As a special case, we can consider to which subunits identity morphisms restrict~\cite[Lemma~1.3]{boyarchenkodrinfeld:idempotents}.

\begin{proposition}\label{prop:supportofobject}
  The following are equivalent for an object $A$ and subunit $s$:
  \begin{enumerate}
    \item[(a)] $s \otimes A \colon S \otimes A \to I \otimes A$ is an isomorphism;
    \item[(b)] there is an isomorphism $S \otimes A \simeq A$; 
    \item[(c)] there is an isomorphism $S \otimes B \simeq A$ for some object $B$;
    \item[(d)] the identity $A \to A$ restricts to $s$.
  \end{enumerate}
\end{proposition}
\begin{proof}
  Trivially (a) $\implies$ (b) $\implies$ (c). For (c) $\implies$ (d): because $s$ is a subunit, $s \otimes S \otimes A$ is an isomorphism, so if $S \otimes B \simeq A$ then also $s \otimes A$ is an isomorphism by Lemma~\ref{lem:subunitsretract}.
  For (d) $\implies$ (a): if $A$ factors through $s \otimes A$, then because $s$ is a subunit $s \otimes S \otimes A$ is an isomorphism, and hence so is $s \otimes A$ by Lemma~\ref{lem:subunitsretract}.
\end{proof}

The following observation is simple, but effective in applications~\cite{enriquemolinerheunentull:space}.

\begin{lemma}\label{lem:supportin}
  Let $s \colon S \to I$ and $t \colon T \to I$ be subunits in a firm category.
  If $f$ restricts to $s$, and $g$ restricts to $t$, then $f \circ g$ and $f \otimes g$ restrict to $s \wedge t$.
\end{lemma}
\begin{proof}
  Straightforward.
\end{proof}

In particular, if $A$ or $B$ restrict to a subunit $s$, then so does any map $A \to B$. 
It also follows that restriction respects retractions: if $e \circ m = \id$, then $m$ restricts to $s$ if and only if $e$ does.

\begin{definition}\label{def:restriction}
  Let $s$ be a subunit in a monoidal category $\cat{C}$. Define the \emph{restriction} of $\cat{C}$ to $s$, denoted by $\cat{C}\restrict{s}$, to be the full subcategory of $\cat{C}$ of objects $A$ for which $s \otimes A$ is an isomorphism.
\end{definition}

\begin{proposition}\label{prop:restrictioncoreflective}
  If $s$ is a subunit in a monoidal category $\cat{C}$, then $\cat{C}\restrict{s}$ is a coreflective monoidal subcategory of $\cat{C}$.
  \[\begin{pic}
    \node (l) at (0,0) {$\cat{C}$};
    \node (r) at (3,0) {$\cat{C}\restrict{s}$};
    \node at (1.5,0) {$\top$};
    \draw[->] ([yshift=2mm]l.east) to ([yshift=2mm]r.west);
    \draw[left hook->] ([yshift=-2mm]r.west) to ([yshift=-2mm]l.east);
  \end{pic}\]
  The right adjoint $\cat{C} \to \cat{C}\restrict{s}$, given by $A \mapsto S \otimes A$ and $f \mapsto S \otimes f$, is also called \emph{restriction} to $s$.
\end{proposition}
\begin{proof}
  First, if $A \in \cat{C}$, note that $S \otimes A$ is indeed in $\cat{C}\restrict{s}$ because $s \otimes S \otimes A$ is an isomorphism as $s$ is a subunit.
  Similarly, $\cat{C}\restrict{s}$ is a monoidal subcategory of $\cat{C}$. 
  Finally, there is a natural bijection 
  \begin{align*}
	\cat{C}(A,B) &\simeq  \cat{C}\restrict{s}(A,S\otimes B)\\
	f &\mapsto (s\otimes f) \circ (s \otimes A)^{-1} \circ \rho_A^{-1}\\
	\lambda_B \circ (s \otimes B) \circ g & \mapsfrom g
  \end{align*}
  for $A \in \cat{C}\restrict{s}$ and $B \in \cat{C}$.
  So restriction is right adjoint to inclusion.
  For monoidality, see~\cite[Theorem~5]{jacobsmandemaker:coreflections}; both functors are (strong) monoidal when $\cat{C}\restrict{s}$ has tensor unit $S$ and tensor product inherited from $\cat{C}$.
\end{proof}

\begin{remark}
The previous result motivates our terminology; a subunit $s$ in $\cat{C}$ is precisely a subobject of $I$ with the property that it may form the tensor unit of a monoidal subcategory of $\cat{C}$, namely $\cat{C}\restrict{s}$.
\end{remark}

{
\begin{example}
  Let $L$ be a semilattice, regarded as a firm category as in {Proposition}~\ref{ex:semilattice}. For a subset $U \subseteq L$ we define $\downset U = \{x \in L \mid x \leq u \text{ for some }u \in U\}$. Then for $s \in L$, the restriction $\cat{C} \restrict{s}$ is the subsemilattice $\downset s = \downset \{s \}$.
\end{example}
}

{
\begin{example}
  Let $L$ be a frame. A subunit in $\mathrm{Sh}(L)$ is just an element $s \in L$, and a morphism $f \colon A \Rightarrow B$ restricts to it precisely when $A(x)=\emptyset$ for $x \not\leq s$.
\end{example}
}

{
\begin{proposition}
  Let $S$ be a nondegenerate firm idempotent ideal of a nondegenerate firm commutative ring $R$.
  Then $\cat{FMod}_R\restrict{S}$ is monoidally equivalent to $\cat{FMod}_S$.
\end{proposition}
}
\begin{proof}
  Send $A$ in $\cat{FMod}_R\restrict{S}$ to $A$ with $S$-module structure $a \cdot s := as$, and send an $R$-linear map $f$ to $f$. This defines a functor $\cat{FMod}_R\restrict{S} \to \cat{FMod}_S$.
  In the other direction, a firm $S$-module $B \simeq B \otimes_S S$ has firm $R$-module structure $(b \otimes s)\cdot r:=b \otimes (sr)$ because $S$ is idempotent, and if $g$ is an $S$-linear map then $g \otimes_S S$ is $R$-linear. This defines a functor $\cat{FMod}_S \to \cat{FMod}_R\restrict{S}$.
  Composing both functors sends a firm $R$-module $A$ to $A \otimes_S S \simeq A \otimes_R R \simeq A$, and a firm $S$-module $B$ to $B \otimes_S S \simeq B$.
\end{proof}

{
\begin{proposition}\label{prop:hilbertmodules:localadjunction}
  For any Hilbert $C_0(X)$-module $A$ and subunit $C_0(U)$ induced by an open subset $U \subseteq X$, the module $A \otimes C_0(U)$ is isomorphic to its submodule
  \[
    A|_U = \{a \in A \mid \inprod{a}{a} \in C_0(U) \}\text,
  \]
  viewing $C_0(U)$ as a closed ideal of $C_0(X)$ via~\eqref{eq:subunithilbertmodules}. Hence in $\HModc$ a morphism $f \colon A \to B$ restricts to this subunit when $\inprod{f(a)}{f(a)} \in C_0(U)$ for all $a \in A$. 
\end{proposition}
}
\begin{proof}
  Write $S=C_0(U)$.
  We first prove that $A \in \HModc\restrict{S}$ if and only if $|a| \in C_0(U)$ for all $a \in A$, where $|a|^2 = \langle a, a \rangle$.
  On the one hand, if $a \in A$ and $f \in S$ then $|a \otimes f|(X \setminus U) = |a| |f| (X \setminus U)=0$. 
  Therefore $|a| \in C_0(U)$ for all $a \in A \otimes S$.
  Because $A \otimes S \simeq A$ is invertible, $|a| \in C_0(U)$ for all $a \in A$.

  On the other hand, suppose that $|a| \in C_0(U)=0$ for all $a \in A$.
  We are to show that the morphism $A \otimes S \to A$ given by $a \otimes f \mapsto af$ is bijective.
  To see injectivity, let $f \in S$ and $a \in A$, and suppose that $af=0$. Then $|a| \cdot |f|=|af|=0$, so for all $x \in U$ either $|a|(x)=0$ or $f(x)=0$. So $|a \otimes f|(U)=0$, and hence $a \otimes f=0$.
  To see surjectivity, let $a \in A$. Then $|a|(x)=0$ for all $x \in X \setminus U$. So $a=\lim af_n$ for an approximate unit $f_n$ of $S$. But that means $a$ is the image of $\lim a \otimes f_n$.
\end{proof}

\begin{remark}
  Restricting $\HModc$ to the subunit $C_0(U)$ for an open subset $U \subseteq X$ gives the full subcategory of modules $A$ with $A = A|_U$. This is nearly, but not quite, $\cat{Hilb}_{C_0(U)}$: any such module also forms a $C_0(U)$-module, but conversely there is no obvious way to extend the action of scalars on a general $C_0(U)$-module to make it a $C_0(X)$-module. There is a so-called \emph{local} adjunction between $\HModc\restrict{C_0(U)}$ and $\cat{Hilb}_{C_0(U)}$, which is only an adjunction when $U$ is clopen~\cite[Proposition~4.3]{clarecrisphigson:adjoint}.	
\end{remark}

Above we restricted along one individual subunit $s$. Next we investigate the structure of the family of these functors when $s$ varies.

\begin{definition}\label{def:gradedmonad} \cite{fujiikatsumatamellies:gradedmonads}
  Let $\cat{C}$ be a category and $(\cat{E}, \otimes, 1)$ a monoidal category. Denote by $[\cat{C}, \cat{C}]$ the monoidal category of endofunctors of $\cat{C}$ with $F\otimes G = G\circ F$.  An \emph{$\cat{E}$-graded monad} on $\cat{C}$ is a lax monoidal functor $T\colon \cat{E} \rightarrow [\cat{C}, \cat{C}]$. More concretely, an $\cat{E}$-graded monad consists of:
  \begin{itemize}
    \item a functor $T\colon \cat{E} \rightarrow [\cat{C}, \cat{C}]$;
    \item a natural transformation $\eta \colon \id[\cat{C}] \Rightarrow T(1)$;
    \item a natural transformation $\mu_{s,t}\colon T(t)\circ T(s) \rightarrow T(s\otimes t)$ for all $s,t$ in $\cat{E}$;
  \end{itemize}
  making the following diagrams commute for all $r,s,t$ in $\cat{E}$.
  \begin{align*}
    \begin{pic}[xscale=2.5, yscale=1.25]
      \node (TsTtTu) at (1,2) {$T(t)\circ T(s)\circ T(r)$};
	  \node (TstTu) at (0,1) {$T(t)\circ T(r\otimes s)$};
	  \node (Tstu1) at (0,0) {$T((r\otimes s)\otimes t)$};
	  \node (Tstu2) at (2,0) {$T(r\otimes (s\otimes t))$};
	  \node (TsTtu) at (2,1) {$T(t\otimes s)\circ T(r)$};
	  \draw[->] (TsTtTu) to node[left, xshift=-2mm]{$\mu_{r,s}\otimes \id[T(t)]$} (TstTu);
	  \draw[->] (TstTu) to node[left]{$\mu_{r\otimes s,t}$} (Tstu1);
	  \draw[->] (Tstu1) to node[above]{$T(\alpha_{r,s,t})$} (Tstu2);
	  \draw[->] (TsTtu) to node[right]{$\mu_{r,s\otimes t}$} (Tstu2);
	  \draw[->] (TsTtTu) to node[right, xshift=1mm]{$\id[T(r)]\otimes\mu_{s,t}$} (TsTtu); 
    \end{pic}
  \end{align*}
  \begin{align*}
    \begin{pic}[xscale=3.5,yscale=.8]
	  \node (idc Ts) at (0,1.5) {$T(s)\circ \id[\cat{C}]$};
	  \node (Ts) at (0,0) {$T(s)$};
	  \node (T1s) at (1.5,0) {$T(1\otimes s)$};
	  \node (T1Ts) at (1.5,1.5) {$T(s)\circ T(1)$};
	  \draw[-, double] (idc Ts) to node[left]{} (Ts);
	  \draw[->] (idc Ts) to node[below]{$\eta\otimes \id[T(s)]$} (T1Ts);
	  \draw[->] (T1Ts) to node[right]{$\mu_{1,s}$} (T1s);
	  \draw[->] (T1s) to node[below]{$T(\lambda_s)$} (Ts);
    \end{pic}
    \\
    \begin{pic}[xscale=3.5,yscale=.8]
	  \node (Ts idc) at (0,1.5) {$\id[\cat{C}]\circ T(s)$};
	  \node (Ts) at (0,0) {$T(s)$};
	  \node (Ts1) at (1.5,0) {$T(s\otimes 1)$};
	  \node (TsT1) at (1.5,1.5) {$T(1)\circ T(s)$};
	  \draw[-, double] (Ts idc) to node[left]{} (Ts);
	  \draw[->] (Ts idc) to node[below]{$\id[T(s)]\otimes \eta$} (TsT1);
	  \draw[->] (TsT1) to node[right]{$\mu_{s,1}$} (Ts1);
	  \draw[->] (Ts1) to node[below]{$T(\rho_s)$} (Ts);
    \end{pic}
  \end{align*}
\end{definition}

\begin{theorem}\label{thm:gradedmonad}
  Let $\cat{C}$ be a monoidal category.
  Restriction is a monad graded over the subunits, when we do not identify monomorphisms representing the same subunit. More precisely, it is an $\cat{E}$-graded monad, where $\cat{E}$ has as objects monomorphisms $s \colon S \rightarrowtail I$ in $\cat{C}$ with $s \otimes S$ an isomorphism, and as morphisms $f \colon s \to t$ those $f$ in $\cat{C}$ with $s = t \circ f$.
\end{theorem}
\begin{proof}
  The functor $\cat{E} \to [\cat{C},\cat{C}]$ sends $s \colon S \rightarrowtail I$ to $(-) \otimes S$, and $f$ to the natural transformation $\id[(-)] \otimes f$.
  The natural transformation $\eta_E \colon E \to E \otimes I$ is given by $\rho_E^{-1}$.
  The family of natural transformations $\mu_{s,t} \colon ((-)\otimes S)\otimes T \rightarrow (-)\otimes(S\otimes T)$ is given by $\alpha_{(-),S,T}$. Associativity and unitality diagrams follow.
\end{proof}

We end this section by giving two characterisations of subunits in terms that are perhaps more well-known. The first characterisation is in terms of idempotent comonads.

\begin{definition}
  A \emph{restriction comonad} on a monoidal category $\cat{C}$ is a monoidal comonad $F \colon \cat{C} \to \cat{C}$:
  \begin{itemize}
    \item whose comultiplication $\delta \colon F \Rightarrow F^2$ is invertible;
    \item whose counit $\varepsilon \colon F \to \id[\cat{C}]$ has a monic unit component $\varepsilon_I \colon F(I) \rightarrowtail I$.
  \end{itemize}
\end{definition}

\begin{proposition}\label{prop:restrictioncomonads}
  Let $\cat{C}$ be a braided monoidal category.
  There is a bijection between subunits in $\cat{C}$ and restriction comonads on $\cat{C}$.
\end{proposition}
\begin{proof}
  If $s \colon S \rightarrowtail I$ is a subunit, then $F(A)=S \otimes A$ defines a comonad by Proposition~\ref{prop:restrictioncoreflective}. Its comultiplication is given by $\delta_A = (\lambda_{S \otimes A} \circ (s \otimes S \otimes A))^{-1}$, by definition being an isomorphism. Its counit is given by $\varepsilon_A = \lambda_A \circ (s \otimes A)$. Because $\rho_I=\lambda_I$, its component $\varepsilon_I = \lambda_I \circ (s \otimes I) = \rho_I \circ (s \otimes I) = s \circ \rho_S$ is monic.
  
  Conversely, if $F$ is a restriction monad, then $\varepsilon_I \colon F(I) \rightarrowtail I$ is a subobject of the tensor unit.
  Writing $\varphi_{A,B} \colon A \otimes F(B) \to F(A \otimes B)$ for the coherence maps, and $\psi_{A,B} = F(\sigma) \circ \varphi_{B,A} \circ \sigma \colon F(A) \otimes B \to F(A \otimes B)$ for its induced symmetric version, the insides of the following diagram commute:
  \[\begin{pic}[xscale=4,yscale=1.2]
    \node (bl) at (0,0) {$F^2(I \otimes I)$};
    \node (b) at (1,0) {$F(I \otimes I)$};
    \node (br) at (2,0) {$F(I \otimes I)$};
    \node (m) at (1,1) {$F^2(I \otimes I)$};
    \node (l) at (0,2) {$F(F(I) \otimes I)$};
    \node (t) at (1,2) {$F(F(I) \otimes I)$};
    \node (tl) at (0,3) {$F(I) \otimes F(I)$};
    \node (tr) at (2,3) {$F(I) \otimes I$};
    \draw[-, double distance=.75mm] (b) to (br);
    \draw[-, double distance=.75mm] (bl) to (m);
    \draw[-, double distance=.75mm] (l) to (t);    
    \draw[->] (tl) to node[above]{$F(I) \otimes \varepsilon_I$} (tr);
    \draw[->] (tl) to node[left]{$\varphi_{F(I),I}$} (l);
    \draw[->] (l) to node[left]{$F(\psi_{I,I})$} (bl);
    \draw[->] (bl) to node[below]{$\delta^{-1}_{I \otimes I}$} (b);
    \draw[->] (b) to node[right]{$\delta_{I \otimes I}$} (m);
    \draw[->] (m) to node[right]{$F(\psi_{I,I}^{-1})$} (t);
    \draw[->] (m) to node[above]{$\varepsilon_{F(I \otimes I)}$} (br);
    \draw[->] (t) to node[below=1mm]{$\varepsilon_{F(I) \otimes I}$} (tr);
    \draw[->] (br) to node[right]{$\psi^{-1}_{I,I}$} (tr);
  \end{pic}\]
  But the long outside path is composed entirely of isomorphisms. Hence $F(I) \otimes \varepsilon_I$ is invertible, and $\varepsilon_I$ is a subunit.

  These two constructions are clearly inverse to each other.
\end{proof}

\begin{remark}\label{rem:subunitcomonads}
  Monoidal comonads on $\cat{C}$ form a category with morphisms of monoidal comonads~\cite{street:formal}.
  This category is monoidal as a subcategory of $[\cat{C},\cat{C}]$.
  The monoidal unit is the identity comonad $A \mapsto A$.
  A subunit is a comonad $F$ with a comonad morphism $\lambda \colon F \Rightarrow \id[\cat{C}]$ whose comultiplication is idempotent, and such that $\lambda_A \colon F(A) \to A$ is monic.
  But by coherence, the latter means that $\varepsilon_I = \lambda_I \colon F(I) \rightarrowtail I$ is monic.
  It follows that subunits in $\cat{C}$ also correspond bijectively to subunits in $[\cat{C},\cat{C}]$ in the same sense as Definition~\ref{def:subunit}, though we have not strictly defined these since the latter category is not braided. See also~\cite[Remark~2.3]{boyarchenkodrinfeld:idempotent}.
\end{remark}

It also follows that restrictions monads automatically satisfy the Frobenius law $\delta^{-1} F \circ F \delta = F \delta^{-1} \circ \delta F$~\cite{heunenkarvonen:monads}, matching the viewpoint in \cite{hines:classicalstructures}.

The second characterisation of subunits $s$ we will give is in terms of the subcategory $\cat{C}\restrict{s}$.

\begin{definition}\label{def:tensorideal}
  Let $\cat{C}$ be a monoidal category.
  A \emph{monocoreflective tensor ideal} is a full replete subcategory $\cat{D}$ such that:
  \begin{itemize}
    \item if $A \in \cat{C}$ and $B \in \cat{D}$, then $A \otimes B \in \cat{D}$;
    \item the inclusion $F \colon \cat{D} \hookrightarrow \cat{C}$ has a right adjoint $G \colon \cat{C} \to \cat{D}$;
    \item the component of the counit at the tensor unit $\varepsilon_I \colon F(G(I))\to I$ is monic;
    \item $F(B) \otimes \varepsilon_I$ is invertible for all $B \in \cat{D}$.
  \end{itemize}
\end{definition}

\begin{proposition}\label{prop:tensorideals}
  Let $\cat{C}$ be a firm category.
  There is a bijection between $\ISub(\cat{C})$ and the set of monocoreflective tensor ideals of $\cat{C}$.
\end{proposition}
\begin{proof}
  A subunit $s$ corresponds to $\cat{C}\restrict{s}$, and a monocoreflective tensor ideal $\cat{D}$ corresponds to $\varepsilon_I$.
  First notice that $\cat{C}\restrict{s}$ is indeed a monocoreflective tensor ideal by Proposition~\ref{prop:restrictioncoreflective}.
  Starting with $s \in \ISub(\cat{C})$ ends up with $s \circ \lambda \colon I \otimes S \rightarrowtail I$, which equals $s$ qua subobject.
  Starting with a monocoreflective tensor ideal $\cat{D}$ ends up with $\{ A \in \cat{C} \mid A \otimes \varepsilon_I \text{ is invertible}\}$.
  We need to show that this equals $\cat{D}$. 
  One inclusion is obvious. For the other, let $A \in \cat{C}$. If $A \otimes \varepsilon_I \colon A \otimes FG(I) \to A \otimes I$ is invertible, then $A \simeq A \otimes F(G(I))$, and so $A \in \cat{D}$ because $\cat{D}$ is a tensor ideal.
\end{proof}

We leave open the question of what sort of factorization systems are induced by monocoreflective tensor ideals~\cite{cassidyhebertkelly:factorization,day:monoidallocalisation}.

\section{Simplicity}\label{sec:simplicity}

Localisation in algebra generally refers to a process that adds formal inverses to an algebraic structure~\cite[Chapter 7]{kashiwara2005categories}. 
This section discusses how to localise all subunits in a monoidal category at once, by showing that restriction is an example of localisation in this sense.

\begin{definition}
	Let $\cat{C}$ be a category and $\Sigma$ a collection of morphisms in $\cat{C}$. 
  A \emph{localisation of $\cat{C}$ at $\Sigma$} is a category $\cat{C}[\Sigma^{-1}]$ and a functor $Q \colon \cat{C} \to \cat{C}[\Sigma^{-1}]$ such that:
	\begin{itemize}
		\item $Q(f)$ is an isomorphism for every $f\in \Sigma$;
		\item for any functor $R\colon\cat{C} \rightarrow \cat{D}$ such that $R(f)$ is an isomorphism for all $f\in\Sigma$, there exists a functor $\overline{R}\colon\cat{C}[\Sigma^{-1}] \rightarrow \cat{D}$ and a natural isomorphism $\overline{R}\circ Q \simeq R$;
    \[\begin{tikzpicture}
      \node (l) at (0,1.3) {$\cat{C}$};
      \node (tr) at (3,1.3) {$\cat{C}[\Sigma^{-1}]$};
      \node (br) at (3,0) {$\cat{D}$};
      \draw[->] (l) to node[above]{$Q$} (tr);
      \draw[->] (l) to node[below]{$R$} (br);
      \draw[->,dashed] (tr) to (br);
      \node at (2,.9) {$\simeq$};
    \end{tikzpicture}\]
		\item precomposition  
		$
			(-)\circ Q \colon \big[\cat{C}[\Sigma^{-1}], \cat{D}\big] \to [\cat{C}, \cat{D}]
		$
		is full and faithful for every category $\cat{D}$.
	\end{itemize}
\end{definition}

\begin{proposition}
	Restriction $\cat{C} \to \cat{C}\restrict{s}$ at a subunit $s$
  is a localisation of $\cat{C}$ at $\{ s \otimes A \mid A\in\cat{C} \}$.
\end{proposition}
\begin{proof}
  Observe that $S\otimes (-)$ sends elements of $\Sigma$ to isomorphisms because $s$ is idempotent. 
	Let $R\colon\cat{C}\rightarrow\cat{D}$ be any functor making $R(s \otimes A)$ an isomorphism for all $A\in\cat{C}$. 
  Define $\overline{R}\colon \cat{C}\restrict{s} \to \cat{D}$ by $A \mapsto R(A)$ and $f \mapsto R(f)$.
  Then 
	\begin{equation*}
		\eta_A = R(\rho_A)\circ R(s\otimes A) \colon R(s\otimes A) \rightarrow R(A)
	\end{equation*}
  is a natural isomorphism.
	It is easy to check that precomposition with restriction is full and faithful.
\end{proof}

The above universal property concerns a single subunit. We now move to localising all subunits simultaneously.

\begin{definition}\label{def:simple}
  A monoidal category is \emph{simple} when it has no subunits but $I$.
\end{definition}

In the words of Proposition~\ref{prop:tensorideals}, a category is simple when it has no proper monocoreflective tensor ideals. Let us now show how to make a category simple.

\begin{proposition}\label{prop:localisation}
  If $\cat{C}$ is a firm category, then there is a universal simple category ${\Simple}(\cat{C})$ with a monoidal functor $\cat{C} \to {\Simple}(\cat{C})$: any monoidal functor $F \colon \cat{C} \to \cat{D}$ into a simple category $\cat{D}$ factors through it via a  unique monoidal functor ${\Simple}(\cat{C})\to \cat{D}$.
  \[\begin{pic}[xscale=10,yscale=1.25]
      \node (tl) at (0,1) {$\cat{C}$};
      \node (tr) at (0.5,1) {${\Simple}(\cat{C})$};
      \node (bl) at (.5,0) {$\cat{D}$};
      \draw[->] (tl) to (tr);
      \draw[->,dashed] (tr) to (bl);
      \draw[->] (tl) to node[below]{$F$} (bl);
  \end{pic}\]
\end{proposition}
\begin{proof}
We proceed by formally inverting the collection of morphisms
\[
    \Sigma = \{ \lambda_A \circ (s \otimes A) \mid A \in \cat{C}, s \in \ISub(\cat{C}) \} \cup \{ A \mid A \in \cat{C} \}\text.
\]
To show that the localisation $\cat{C}[\Sigma^{-1}]$ of $\Sigma$ exists we will show that $\Sigma$ admits \emph{a calculus of right fractions}~\cite{gabrielzisman:calculusoffractions}. Firstly, $\Sigma$ contains all identities and is closed under composition, since the composition of $\lambda_A \circ (A \otimes t)$ and $\lambda_{A \otimes T} \circ (A \otimes T \otimes s)$ is simply $\lambda_A \circ (A \otimes (s \wedge t))$. It remains to show that: 
\begin{itemize}
 \item for morphisms $s \colon A \to C$ in $\Sigma$ and $f \colon B \to C$ in $\cat{C}$, there exist morphisms $t \colon P \to B$ in $\Sigma$ and $g \colon P \to A$ in $\cat{C}$ such that $g\circ s = t \circ f$;
      \[\begin{tikzpicture}[scale=1.5]
        \node (tl) at (0,1) {$\bullet$};
        \node (tr) at (1,1) {$\bullet$};
        \node (bl) at (0,0) {$\bullet$};
        \node (br) at (1,0) {$\bullet$};
        \draw[->] (bl) to node[below]{$f$} (br);
        \draw[->] (tr) to node[right]{$s \in \Sigma$} (br);
        \draw[->,dashed] (tl) to node[left]{$\Sigma \ni t$} (bl);
        \draw[->,dashed] (tl) to node[above]{$g$} (tr);
      \end{tikzpicture}\]
    \item if a morphism $t \colon C \to D$ in $\Sigma$ and $f,g \colon B \to C$ in $\cat{C}$ satisfy $t \circ f = t \circ g$, then $f \circ s = g \circ s$ for some $s \colon A \to B$ in $\Sigma$.
\end{itemize}

It suffices to merely consider $\{\lambda_A \circ (s \otimes A) \mid A \in \cat{C}, s \in \ISub(\cat{C}) \}$ by~\cite[Remark~3.1]{fritz:categoriesoffractions}. The first, also called the \emph{right Ore condition}, is satisfied by bifunctoriality of the tensor:
  \[\begin{tikzpicture}[xscale=4,yscale=1.2]
    \node (tl) at (0,1) {$S \otimes A$};
    \node (tr) at (1,1) {$S \otimes B$};
    \node (l) at (0,0) {$I \otimes A$};
    \node (r) at (1,0) {$I \otimes B$};
    \node (bl) at (0,-1) {$A$};
    \node (br) at (1,-1) {$B$};
    \draw[->] (bl) to node[below]{$f$} (br);
    \draw[->] (l) to node[below]{$I \otimes f$} (r);
    \draw[->,dashed] (tl) to node[above]{$S \otimes f$} (tr);
    \draw[->] (tr) to node[right]{$s \otimes B$} (r);
    \draw[->] (r) to node[right]{$\rho_B$} (br);
    \draw[->,dashed] (tl) to node[left]{$s \otimes A$} (l);
    \draw[->,dashed] (l) to node[left]{$\rho_A$} (bl);
  \end{tikzpicture}\]
For the second, suppose that $(s \otimes B) \circ f = (s \otimes B) \circ g$. 
Then applying $S \otimes (-)$ and using that $S \otimes s$ is invertible, it follows that $S \otimes f=S \otimes g$. But then 
  \begin{align*}
    f \circ \lambda_A \circ (s \otimes A)
    &  = \lambda_{SB} \circ (s \otimes S \otimes B) \circ (S \otimes f) \\
    & = \lambda_{SB} \circ (s \otimes S \otimes B) \circ (S \otimes g)
       = g \circ \lambda_A \circ (s \otimes A)\text,
  \end{align*}
  so the second requirement is satisfied.
 As a result, $\cat{C}[\Sigma^{-1}]$ exists; an easy constuction may be found in~\cite{fritz:categoriesoffractions}. It satisfies the universal property of localisation on the nose.
 {We define $\Simple(\cat{C}) = \cat{C}[\Sigma^{-1}]$.}
 Moreover, the functor $\cat{C} \to {\Simple}(\cat{C})$ is monoidal because the class $\Sigma$ is closed under tensoring with objects of $\cat{C}$ by construction~\cite[Corollary~1.4]{day:monoidallocalisation}. Finally, notice that ${\Simple}(\cat{C})$ is simple by construction.
\end{proof}

\section{Support}\label{sec:support}


When a morphism $f$ restricts to a given subunit $s$, we might also say that $f$ `has support in' $s$. 
Indeed it is natural to assume that each morphism in our category comes with a canonical least subunit to which it restricts, which we may call its support. This is the case in a topos, for example, but in general requires extra structure. 

Write $\mors{\cat{C}}$ for the braided monoidal category whose objects are morphisms $f \in \cat{C}$, with $f \otimes g$ defined as in $\cat{C}$, tensor unit $I$, and a unique morphism $f \to g$ whenever ($g$ restricts to $s$) $\implies$ ($f$ restricts to $s$). 

\begin{definition}\label{def:supportdatum}
  A \emph{support datum} on a firm category $\cat{C}$ is a functor $F \colon \mors{\cat{C}} \to L$ into a complete lattice $L$ satisfying
  \begin{equation}\label{eq:supportdatum}
    F(f) = \bigwedge \big\{ F(s) \colon s \in \ISub(\cat{C}) \mid f \text{ restricts to } s \big\}
  \end{equation}
  for all morphisms $f$ of $\cat{C}$. 
  A \emph{morphism of support data} $F \to F'$ is one of complete lattices $G \colon L \to L'$ with $G \circ F = F'$.
\end{definition}

\begin{lemma}
  If $F \colon \mors{\cat{C}} \to L$ is a support datum, and $f, g$ morphisms in $\cat{C}$:
  \begin{itemize}
    \item $F(f) = \bigwedge \{ F(A) \mid A \in \cat{C},\, f \text{ factors through $A$} \}$;
    \item $F(f \otimes g) \leq F(f) \wedge F(g)$ for all $f, g$; so $F$ is colax monoidal.
  \end{itemize}
\end{lemma}
This notion of support via objects is similar to that of~\cite{balmer:spectrum,kockpitsch:pointfree,joyal:chevalleytarski}.
\begin{proof}
  For the first statement, it suffices to show that $f$ restricts to a subunit $s$ iff it factors through some object $A$ which does. But if $f$ factors through $A$ then $f=g \circ A \circ h$ for some $g, h$ and so if $A$ restricts to $s$ so does $f$. Conversely if $f \colon B \to C$ restricts to $s$ it factors over $S \otimes C$, which always restricts to $s$.

  For the second statement, Note that $F(I) \leq 1$ always, so colax monoidality reduces to the rule above. But if $f$ restricts to $s$ then so does $f \otimes g$. Hence $F (f \otimes g) \leq F(f)$, and $F(f \otimes g) \leq F(g)$ similarly.
\end{proof}

Most features of support data follow from the associated map $\ISub(\cat{C}) \to L$. 

\begin{proposition} \label{prop:suppasfunctor}
  Let $\cat{C}$ be a firm category and $L$ a complete lattice. Specifying a support datum $F \colon \mors{\cat{C}} \to L$ is equivalent to specifying a monotone map $\ISub(\cat{C}) \to L$. 
\end{proposition}
\begin{proof}
  In $\mors{\cat{C}}$ there is a morphism $s \to t$ between subunits $s$ and $t$ precisely when $s \leq t$. Hence any support datum restricts to a monotone map $\ISub(\cat{C}) \to L$.

  Conversely, let $F$ be such a map and extend it to arbitrary morphisms by~\eqref{eq:supportdatum}. Both definitions of $F$ agree on subunits $s$ since a subunit restricts to another one $t$ precisely when $s \leq t$, so that $F(s)=\bigwedge \{F(t) \mid s \leq t\}$. Finally, for functoriality suppose there exists a morphism $f \to g$ in $\mors{\cat{C}}$. If this holds then whenever $g$ restricts to $s$ then so does $f$, so that $F(f) \leq F(g)$. 
\end{proof}

This observation provides examples of support data. Recall that the free complete lattice on a semilattice $L$ is given by its collection $D(L)$ of downsets $U = \downset U \subseteq L$ under inclusion, via the embedding $x \mapsto \downset x$~\cite[II.1.2]{johnstone:stonespaces}. 

\begin{proposition}\label{prop:initialsupport}
  Any firm category $\cat{C}$ has a canonical support datum, valued in $D(\ISub(\cat{C}))$, given by
  \begin{equation} \label{eq:suppfcanonical}
    \suppinit(f) = \{s \in \ISub(\cat{C}) \mid \text{$f$ restricts to $t$} \implies s \leq t \}\text.
  \end{equation}
  Moreover, $\suppinit$ is initial: any support datum factors through it uniquely.
  \[\begin{tikzpicture}[xscale=4,yscale=1.5]
    \node (tl) at (0,1) {$\mors{\cat{C}}$};
    \node (tr) at (1,1) {$D(\ISub(\cat{C}))$};
    \node (br) at (1,0) {$L$};
    \node (t) at (1.5,1) {$\{s_i\}$};
    \node (b) at (1.5,0) {$\bigvee F(s_i)$};
    \draw[->] (tl) to node[above]{$\suppinit$} (tr);
    \draw[->] (tl) to node[below]{$F$} (br);
    \draw[->,dashed] (tr) to (br);
    \draw[|->] (t) to (b);
  \end{tikzpicture}\]
\end{proposition}
This generalises~\cite{balmer:spectrum,balmerfavi:telescope,balmerkrausestevenson:smashing} from triangulated categories to firm ones.
\begin{proof}
  Extend the embedding $L \to D(L)$ to a support datum via Proposition~\ref{prop:suppasfunctor}. Initiality is immediate by freeness of $D(L)$, with~\eqref{eq:suppfcanonical} coming from the description of meets in terms of joins in a complete lattice.
\end{proof}

Rather than require extra data, it would be desirable to define support internally to the category. If $\cat{C}$ has the property that $\ISub(\cat{C})$ is already a complete lattice (or frame), then it indeed comes with a support datum given by the identity on $\ISub(\cat{C})$. We may then define \emph{the support} of a morphism as 
\[
  \supp(f) = \bigwedge \big\{ s \in \ISub(\cat{C}) \mid f \text{ restricts to } s \big\}\text.
\]
Note that $\supp(f) = \bigvee \suppinit(f)$. It therefore follows from Proposition~\ref{prop:initialsupport} that $\supp$ also has a universal property: if $\ISub(\cat{C})$ is already a complete lattice, any support datum $F$ factors through $\supp$ via a semilattice morphism.
Therefore, in the case of a topos, $\supp(A)$ is the factorisation of a morphism $A \to 1$ into a strong epimorphism and a monomorphism.

\begin{example}
  Let $L$ be a frame and consider $\mathrm{Sh}(L)$. A morphism $f \colon A \Rightarrow B$ has $\suppinit(f) = \downset \{ t \mid A(t)\neq\emptyset\}$, and $\supp(f) = \bigwedge \{ s \mid A(s) \neq\emptyset \}$.
\end{example}

\begin{example}
  In $\HModc$ the collection of subunits forms a frame, and each morphism $f \colon A \to B$ has $\supp(f) = C_0(U_f)$, where 
  \[
    U_f = \{x \in X \mid \inprod{f(a)}{f(a)}(x) \neq 0 \text{ for some } a \in A\}\text.
  \]
  Letting $L$ be the totally ordered set of cardinals below $|X|$, we may define another support datum by $F(f) = |U_f| \in L$. 
\end{example}

In the remaining sections we turn to categories coming with such an intrinsic spatial structure. First, the following example shows that, even in case $\ISub(\cat{C})$ is a frame, our notion of support differs from that of~\cite[Definition~3.1(SD5)]{balmer:spectrum} and~\cite[Definition~3.2.1(5)]{kockpitsch:pointfree}: without further assumptions, a support datum is only colax monoidal.

{
\begin{proposition} 
  There is a firm category $\cat{C}$ for which $\ISub(\cat{C})$ is a frame but
  \[
    \supp(f) \otimes \supp(g) \neq \supp(f \otimes g)\text.
  \]
\end{proposition}
}
\begin{proof}
  Let $Q$ be the commutative unital quantale with elements $0 \leq \varepsilon \leq 1$, with unit $1$ and satisfying $0 = 0 \cdot 0 = 0 \cdot \varepsilon = \varepsilon \cdot \varepsilon$. Then the frame of subunits is $\ISub(Q) = \{0,1\}$, and $\varepsilon$ satisfies $\supp(\varepsilon) = 1$ whereas $\supp(\varepsilon \cdot \varepsilon) = 0$. 
\end{proof}

\section{{Locale-based categories}}\label{sec:spatial}

In our main examples, the subunits satisfy extra properties over being a mere semilattice, and they interact universally with the rest of the category. First, they often satisfy the following property. 

\begin{definition}\label{def:stiff}
  A category is \emph{stiff} when it is braided monoidal and
  \begin{equation} \label{eq:stiff-pullback}
  \begin{pic}[xscale=4,yscale=1.5]
    \node (tl) at (0,1) {$S \otimes T \otimes X$};
    \node (tr) at (1,1) {$T \otimes X$};
    \node (bl) at (0,0) {$S \otimes X$};
    \node (br) at (1,0) {$X$};
    \draw[>->] (tl) to node[above]{$s \otimes T \otimes X$} (tr);
    \draw[>->] (tl) to node[left]{$S \otimes t \otimes X$} (bl);
    \draw[>->] (tr) to node[right]{$t \otimes X$} (br);
    \draw[>->] (bl) to node[below]{$s \otimes X$} (br);
    \draw (.1,.7) to (.15,.7) to (.15,.825);
  \end{pic}
  \end{equation}
  is a pullback of monomorphisms for all objects $X$ and subunits $s,t$.
\end{definition}

Any stiff category is firm: take $X=I$ and recall that pullbacks of monomorphisms are monomorphisms. More strongly, subunits often come with joins satisfying the following.

\begin{definition}\label{def:universalfinitejoins}
  Let $\cat{C}$ be a braided monoidal category. We say that $\cat{C}$ has \emph{universal finite joins} of subunits when it has an initial object $0$ whose morphism $0 \to I$ is monic, with $X \otimes 0 \simeq 0$ for all objects $X$, and $\ISub(\cat{C})$ has finite joins such that each diagram
  \begin{equation} \label{eq:pullback-pushout}
  \begin{pic}[xscale=4,yscale=1.5]
    \node (tl) at (0,1) {$S \otimes T \otimes X$};
    \node (tr) at (1,1) {$T \otimes X$};
    \node (bl) at (0,0) {$S \otimes X$};
    \node (br) at (1,0) {$(S \vee T) \otimes X$};
    \draw[>->] (tl) to node[above]{} (tr);
    \draw[>->] (tl) to node[left]{} (bl);
    \draw[>->] (tr) to node[right]{} (br);
    \draw[>->] (bl) to node[below]{} (br);
    \draw (.1,.7) to (.15,.7) to (.15,.825);
    \draw (.95,.3) to (.90,.3) to (.90,.175);
  \end{pic}
  \end{equation}
  is both a pullback and pushout of monomorphisms, where each morphism is the obvious inclusion tensored with $X$ as in~\eqref{eq:stiff-pullback}.
\end{definition}

\begin{lemma} \label{lem:lat}
  Let $\cat{C}$ be braided monoidal with universal finite joins of subunits. Then $\cat{C}$ is stiff and $\ISub(\cat{C})$ is a distributive lattice with least element $0$.
\end{lemma}
\begin{proof}
  For stiffness, take $t = \id[I]$ to see that each morphism $s \otimes X$ is monic. Then since $(s \vee t) \otimes X$ is monic it follows easily that each diagram~\eqref{eq:stiff-pullback} is a pullback. By assumption $0 \to I$ is indeed a subunit. Finally it follows from~\eqref{eq:pullback-pushout} with $X=R$ that subunits $R, S, T$ satisfy $(S \vee T) \wedge R = (S \wedge R) \vee (T \wedge R)$.
\end{proof}

{
\begin{proposition} \label{prop:coherentUnivUnion}
  Any coherent category $\cat{C}$ forms a cartesian monoidal category with universal finite joins of subunits. 
\end{proposition}
}
\begin{proof}
  Each partial order $\Sub(A)$ is a distributive lattice, and for subobjects $S, T \rightarrowtail A$ each diagram~\eqref{eq:pullback-pushout} with $\wedge$ replacing $\otimes$ and $X=1$ is indeed both a pushout and pullback~\cite[A1.4.2, A1.4.3]{johnstone:elephant}. Moreover in such a category each functor $X \times (-)$ preserves these pullbacks, since limits commute with limits, and preserves finite joins and hence these pushouts since each functor $(\pi_2)^* \colon \Sub(A) \to \Sub(X \times A)$ does so by coherence of $\cat{C}$. 
\end{proof}

To obtain arbitrary joins of subunits from finite ones, it will suffice to also have the following. Recall that a subset $U$ of a partially ordered set is (upward) \emph{directed} when any $a,b \in U$ allow $c \in U$ with $a \leq c \geq b$. A \emph{preframe} is a semilattice in which every directed subset has a supremum, and finite meets distribute over directed suprema.

By a \emph{directed colimit of subunits} we mean a colimit of a diagram $D \colon \cat{J} \to \cat{C}$, for which $\cat{J}$ is a directed poset, all of whose arrows are inclusions $S_i \rightarrowtail S_j$ between a collection of subunits $s_i \colon S_i \to I$. In particular $D$ has a cocone given by these subunits, inducing a morphism $\colim D \to I$ if a colimit exists.
 
\begin{definition}\label{def:universaldirectedjoins}
  A stiff category $\cat{C}$ has \emph{universal directed joins} of subunits when it has directed colimits of subunits, each of whose induced arrow $\colim S \to I$ is again a subunit, and these colimits are preserved by each functor $X \otimes (-)$.
\end{definition}

\begin{lemma} \label{lem:preframe}
  If a stiff category $\cat{C}$ has universal directed joins of subunits, then $\ISub(\cat{C})$ is a preframe.
\end{lemma}
\begin{proof}
  Any directed subset $U \subseteq \ISub(\cat{C})$ induces a diagram $U \to \cat{C}$, and its colimit is by assumption a subunit which is easily seen to form a supremum of $U$. Taking $X$ to be a subunit shows that $\wedge$ distributes over directed suprema.
\end{proof}

\begin{example}
  Any preframe $L$, regarded as a monoidal category under $(\wedge,1)$, has universal directed joins. 
\end{example}

The rest of this section shows that the subunits of a category have a spatial nature when it has both types of universal joins above. We unify Definitions~\ref{def:universalfinitejoins} and~\ref{def:universaldirectedjoins} as follows.
Let $\cat{C}$ be a braided monoidal category and $U \subseteq \ISub(\cat{C})$ a family of subunits. For any object $X$, write $D(U,X)$ for the diagram of objects $S \otimes X$ for $s \in U$ and all morphisms $f \colon S \otimes X \to T \otimes X$ satisfying $(t \otimes X) \circ f = s \otimes X$.
If $\cat{C}$ is stiff, there is a unique such $f$ for $s$ and $t$.
\[
\begin{tikzpicture}[xscale=6,yscale=1]
    \node (tl) at (-0.4,1) {$S \otimes X$};
    \node (tr) at (0.4,1) {$T \otimes X$};
    \node (b) at (0,0) {$X$};
    \draw[>->,dashed] (tl) to node[above]{} (tr);
    \draw[>->] (tl) to node[left=3mm]{$s \otimes X$} (b);
    \draw[>->] (tr) to node[right=2mm]{$t \otimes X$} (b);
\end{tikzpicture}
\]
Call such a set $U$ of subunits \emph{idempotent} when $U = U \otimes U := \{s \wedge t \mid s, t \in U\}$.

\begin{definition}\label{def:spatial}
  A category $\cat{C}$ is {\emph{locale-based}} when it is stiff, $\ISub(\cat{C})$ is a frame, and the canonical maps $S \otimes X \to (\bigvee U) \otimes X$ form a colimit of $D(U,X)$ for each idempotent $U \subseteq \ISub(\cat{C})$ and $X \in \cat{C}$.
\end{definition}

Let us now see how this combines our earlier notions. In any poset $P$, an \emph{ideal} is a downward closed, upward directed subset. Let us call a subset $U \subseteq P$ \emph{finitely bounded} when it has a finite set of maximal elements. 
If $U$ is downward closed then equivalently it is finitely generated: $U=\downset\{x_1, \dots, x_n\}$.

\begin{proposition} \label{prop:unify}
  A category $\cat{C}$ has universal finite (directed) joins if and only if
  $\ISub(\cat{C})$ has finite (directed) joins, and $D(U,X)$ has colimit $S \otimes X \to (\bigvee U) \otimes X$ for each idempotent $U \subseteq \ISub(\cat{C})$ that is finitely bounded (directed).
\end{proposition}
\begin{proof}
  First consider finite joins. A colimit of $D(\emptyset, X)$ is precisely an initial object and the conditions on $0$ in both cases are equivalent to $0 \to I$ being a subunit with $0 \otimes X \simeq 0$ for all $X$. 
  Moreover in any stiff category it is easy to see that cocones over the top left corner of \eqref{eq:pullback-pushout} correspond to those over $D(\downset\{s, t\},X)$. (See also Lemma~\ref{lem:idempotentsetsofidempotentsubunits} below.) Hence the properties above provide each diagram with a colimit $(S \vee T) \otimes X$, and so $\cat{C}$ with universal finite joins. 

  Conversely, suppose that $\cat{C}$ has universal finite joins. For any idempotent $U$ we claim that any cocone $c_s$ over $D(U,X)$ extends to one over $D(V,X)$, where $V = \{s_1 \vee \dots \vee s_n \mid s_i \in U\}$. Indeed for any $s, t \in U$ the following diagram commutes, giving $c_{s \vee t}$ as the unique mediating morphism.
  \[\begin{tikzpicture}[xscale=5,yscale=1.5]
    \node (tl) at (0,1) {$S \otimes T \otimes X$};
    \node (tr) at (0.7,1) {$T \otimes X$};
    \node (bl) at (0,0) {$S \otimes X$};
    \node (br) at (0.7,0) {$(S \vee T) \otimes X$};
    \node (x) at (1.2,-.8) {$C$};
    \draw[>->] (tl) to node[above]{} (tr);
    \draw[>->] (tl) to node[left]{} (bl);
    \draw[->] (tr) to node[right]{} (br);
    \draw[->] (bl) to node[below]{} (br);
    \draw[>->] (tr) to[out=-30,in=90,looseness=.7] node[right]{$c_t$} (x);
    \draw[>->] (bl) to[out=-60,in=180] node[below]{$c_s$} (x);
    \draw[->,dashed] (br) to node[above]{$c_{s \vee t}$} (x);
    \draw (.65,.3) to (.60,.3) to (.60,.175);
  \end{tikzpicture}\]
  Similarly define morphisms $c_{s_1 \vee \dots \vee s_n}$ for arbitrary elements of $V$; these form a cocone. Hence $\colim D(U,X) = \colim D(V,X)$. But if $U$ is bounded by some $s_1, \dots, s_n$ then clearly $\colim D(V,X) = (s_1 \vee \dots \vee s_n) \otimes X$ and we are done. 

  Next, consider directed joins. Let $D$ be a directed diagram of inclusions between elements of $U \subseteq \ISub(\cat{C})$. Then $U$ must be directed and therefore $V = \{s_1 \wedge \dots \wedge s_n \mid s_i \in U\}$ is idempotent and directed. Moreover, for each object $X$, any cocone $c_s$ over $D \otimes X$ extends to one over $D(V,X)$: for any $s \in V$, let $s \leq t \in U$ and set $c_s = c_t \circ (x \otimes \id[X])$ where $x \colon S \to T$ is the inclusion. Since $R = \bigvee V$ has $R \otimes X= \colim D(V,X)$  then $R \otimes X=\colim (D \otimes X)$ as required. 

  Conversely, suppose $\cat{C}$ has universal directed joins. Then $\ISub(\cat{C})$ is a preframe by Lemma~\ref{lem:preframe}. If $U \subseteq \ISub(\cat{C})$ is directed and idempotent then for each $X$ we have $R \otimes X = \colim \big(D(U,I) \otimes X\big)$, where $R = \bigvee U$. But any cocone over $D(U,X)$ certainly also forms one over $D(U,I) \otimes X$, and so $R \otimes X = \colim D(U,X)$ also. 
\end{proof}

\begin{corollary} \label{cor:spatialiffboth}
  A category is {locale-based} if and only if it has universal finite and directed joins of subunits.
\end{corollary}
\begin{proof}
  Proposition~\ref{prop:unify} proves one direction. In the other direction, suppose $\cat{C}$ has universal finite and directed joins of subunits. Then $\ISub(\cat{C})$ is a frame by Lemmas~\ref{lem:lat} and~\ref{lem:preframe}, since a poset is a frame precisely when it is a preframe and a distributive lattice. 
  Let $U \subseteq \ISub(\cat{C})$ be idempotent. Then $V = \{s_1 \vee \dots \vee s_n \mid s_i \in U\}$ is idempotent by distributivity, as well as directed, so that $\colim D(V,X) = (\bigvee V) \otimes X$ exists for any $X$. But $\colim D(U,X) = \colim D(V,X)$ as in the proof of Proposition~\ref{prop:unify}.
\end{proof}

The previous corollary justifies saying that a category simply \emph{has universal joins} of subunits when it is {locale-based}. The rest of this section shows that our main examples are {locale-based}. 

\begin{example} \label{ex:frameisspatial}
  Any commutative unital quantale $Q$ is {locale-based} when regarded as a category as in {Proposition}~\ref{ex:quantale}; in particular so is any frame under tensor $\wedge$. Indeed that {proposition} showed that $\ISub(Q)$ is a frame, and for any $U \subseteq \ISub(Q)$ and $x \in Q$ we have $\colim D(U,x) = \bigvee_{s \in U} sx = (\bigvee_{s \in U} s)x$.
\end{example}

{
\begin{proposition} \label{prop:topos-spatial}
  Any cocomplete Heyting category $\cat{C}$ is {locale-based} under cartesian products. This includes all cocomplete toposes, such as Grothendieck toposes.
\end{proposition}
}
\begin{proof}
  Since a Heyting category is coherent, it has universal finite joins by {Proposition}~\ref{prop:coherentUnivUnion}, with each change of base functor having a right adjoint and so preserving arbitrary joins of subobjects. In any cocomplete regular category with this property, for any directed diagram $D$ and any cocone $C$ over $D$ all of whose legs are monic, the induced map $\colim D \to C$ is again monic~\cite[Corollary II.2.4]{grillet1971regular}. Hence whenever $U$ is directed, so is each map $\colim D(U,X) \to X$, ensuring that $\colim D(U,X) = \bigvee_{s \in U} s \times X$ is in $\Sub(X)$. Since each functor $X \times (-)$ now preserves arbitrary joins of subobjects furthermore $\bigvee_{s \in U} s \times X = \colim D(U,I) \times X$, establishing universal directed joins. 
\end{proof}

Next we consider Hilbert modules. In general $\HModc$ is finitely cocomplete but not cocomplete, and so lacks directed colimits by~\cite[IX.1.1]{maclane:categorieswork}; this follows from~\cite[Example 2.3~(9)]{adamek1994locally} by taking $X$ to be trivial and so reducing to the category of Hilbert spaces and {nonexpansive} linear maps. Nonetheless, we have the following.

{
\begin{proposition}
$\HModc$ is {locale-based}.
\end{proposition}
}
\begin{proof}
  Throughout this proof we again identify $C_0(U)$ with the submodule~\eqref{eq:subunithilbertmodules} of $C_0(X)$, and identify the module $A \otimes C_0(U)$ with $A|_U$, for open $U \subseteq X$.

  First let us show that $\HModc$ has universal finite joins of subunits. For open subsets $U, V \subseteq X$, and any Hilbert $C_0(X)$-module $A$, consider the diagram of inclusions between $A|_{U \cap V}$, $A|_U$, $A|_V$ and $A|_{U \cup V}$. It is easily seen to be a pullback, since $A|_{U \cap V} = A|_U \cap A|_V$ as subsets of $A$. We verify that it is also a pushout. Since any morphism $A_{U \cup V} \to B$ restricts to $C_0(U \cup V)$, it suffices to assume that $X = U \cup V$. We claim that
  \[
    C_0(U) + C_0(V) = \{g_U + g_v \in C_0(X) \mid g_U \in C_0(U), g_V \in C_0(V)\}
  \]
  is a dense submodule of $C_0(X)$.
  To see this, let $g \in C_0(X)$ and $\varepsilon > 0$, and $K$ be compact with $|g(x)| \geq \epsilon \implies x \in K$. Urysohn's lemma for locally compact Hausdoff spaces~\cite[2.12]{rudin2006real} produces $h \in C_0(U)$ such that $|h(x)| \leq |g(x)|$ for $x \in U$ and $h(x)=g(x)$ for $x \in K \setminus V$. Then $|(g-h)(x)| \geq 2 \varepsilon \implies x \in L$ for some compact $L \subseteq K \cap V$. Again there is $k \in C_0(V)$ with $|k(x)| \leq |g(x)|$ for all $x \in V$ and $k(x)=(g-h)(x)$ for $x \in L$. By construction $\|g-h-k\| \leq 4 \varepsilon$, establishing the claim. It follows also that 
  \[
    A|_U + A|_V = \{a_U + a_V \mid a_U \in A|_U, a_V \in A|_V\}
  \] 
  is dense in $A$, since $A \cdot C_0(X) = \{a \cdot g \mid g \in C_0(X) \}$ is so too~\cite[p5]{lance:hilbertmodules}. 

  Now suppose $f_U \colon A|_U \to B$ and $f_V \colon A|_V \to B$ agree on $A|_{U \cap V}$. Then for $a=a_U + a_V$ with $a_U \in A|_U$ and $a_V \in A|_V$, the assignment
  \[
    f(a) = f_U(a_U) + f_V(a_V)
  \]
 is a well-defined $A$-linear map. Hence it extends to a unique map $f \colon A \to B$ which is by definition the unique factorisation of $f_U$ and $f_V$ through the diagram. 

  Now we must check that $f$ is {nonexpansive} when $f_U$ and $f_V$ are. Let $x \in X$, and without loss of generality say $x \in U$. 
   Urysohn's lemma again produces $g \in C_0(U)$ with $g(x) = 1 = \|g\|$. Now $a \cdot g \in A|_U$ for any $a \in A$. So, writing $|a|^2(x)$ for $|\inprod{a}{a}(x)|$, we find
  \[
    |f(a)|(x) 
    = |f(a)\cdot g| (x)  
    \leq \|f(a) \cdot g\|
    = \|f_U(a) \cdot g\|
    \leq \|a\| \|g\|
    \leq \|a\|
  \]
  using $\|f_U\|\leq 1$. Since $x$ was arbitrary, also $\|f\|\leq 1$.

  Next, let us consider universal directed joins of subunits. For this, let $W$ be a directed family of open sets in $X$; again it suffices to assume $X = \bigcup W$. We claim that 
  \[
    \bigcup_{U \in W} C_0(U) = 
    \{g \in C_0(X) \mid g \in C_0(U) \text{ for some }U \in W\}
  \]
  is a dense submodule of $C_0(X)$. Again let $g \in C_0(X)$ and $\varepsilon > 0$, and let $K$ be compact with $|g(x)| \geq \epsilon \implies x \in K$. Since $K$ is compact and $W$ is directed, $K \subseteq U$ for some $U \in W$. Urysohn again provides $h \in C_0(U)$ with $|h(x)| \leq |g(x)|$ for all $x \in U$ and $h(x)=g(x)$ for $x \in K$. Then $|g - h|(x) \leq |g(x)| + |h(x)| \leq 2 \varepsilon$ for $x \in X \setminus K$ and so, since $g$ and $h$ agree on $K$, we have $\|g-h\| \leq 2 \varepsilon$, establishing the claim. Similarly, for any Hilbert module $A$, since $A \cdot C_0(X)$ is dense in $A$, so is $\bigcup_{U \in W} A|_U$.

  Finally, let $f_U \colon A|_U \to B$ be a cocone over $D(W,A)$. It suffices to show that there is a unique $f \colon A \to B$ with $f(a) = f_U(a)$ for all $a \in A|_U$. But any $a \in A$ has $a = \lim (a_n)^{\infty}_{n=1}$ with each $a_n \in A|_{U_n}$ for some $U_n$. By directedness we may assume $U_n \subseteq U_{n+1}$ for all $n$. Then $f \colon A \to B$ must satisfy $f(a) = \lim f_{U_n}(a_n)$, making $f$ unique. Additionally, this limit is always well-defined since $a_n$ is a Cauchy sequence and so for $n \leq m$:
  \[
    \|f_{U_n}(a_n) - f_{U_m}(a_m)\| = \|f_{U_m}(a_n - a_m)\| \leq \|a_n - a_m \| 
  \]
  and $f_{U_n}(a_n)$ is also a Cauchy sequence. Clearly $f$ is $A$-linear and $\|f\|\leq 1$.
\end{proof}

\section{Universal joins from colimits}\label{sec:universaljoins}

This section characterises each of the notions of universal joins purely categorically, without order-theoretic assumptions on $\ISub(\cat{C})$. Instead, they will be cast solely in terms of the diagrams $D(U,X)$. When we turn to completions in the next sections, we can therefore use the diagrams $D(U,X)$ themselves as formal joins to add.

\begin{lemma}\label{lem:idempotentsetsofidempotentsubunits}
  Let $\cat{C}$ be a stiff category. If $U \subseteq \ISub(\cat{C})$ is idempotent, then any cocone over $D(U,X)$ extends uniquely to one over $D(\downset U, X)$. 
\end{lemma}
  Therefore, $\cat{C}$ has colimits of $D(U,X)$ for all downward-closed $U \subseteq \ISub(\cat{C})$ if and only if it has them for idempotent $U$.
\begin{proof}
  Let $U$ be idempotent and consider a cocone $c_s \colon S \otimes X \to X$ over $D(U,X)$. 
  Let $r \in \downset U$, say $r = s \circ f$ for $s \in U$ and $f \colon R \to S$.
  Define $c_r = c_s \circ (f \otimes X) \colon R \otimes X \to X$. 
  This is clearly the only possible extension of $c_s$ to $D(\downset U, X)$. We will prove that it is a well-defined cocone.
  Suppose $r' \in \ISub(\cat{C})$ satisfies $r' \leq s'$ for $s' \in U$, and $r \otimes X = (r' \otimes X) \circ g$.
  Then the marked morphism in the following diagram is an isomorphism:
  \[\begin{pic}[xscale=5]
    \node (tl) at (-1,3) {$R \otimes X$};
    \node (tr) at (1,3) {$R' \otimes X$};
    \node (um) at (0,2.25) {$R \otimes R' \otimes X$};
    \node (lml) at (-1,1) {$S \otimes X$};
    \node (lmc) at (0,1) {$S \otimes S' \otimes X$};
    \node (lmr) at (1,1) {$S' \otimes X$};
    \node (b) at (0,0.25) {$X$};
    \draw[>->] (tl) to node[above]{$g$} (tr);
    \draw[>->] (tl) to (lml);
    \draw[>->] (tr) to (lmr);
    \draw[>->] (um) to (lmc);
    \draw[>->] (um) to node[below]{$r\otimes R'\otimes X$} (tr);
    \draw[>->] (um) to node[right=8mm]{$\simeq$} node[below=1mm]{$R \otimes r' \otimes X$} (tl);
    \draw[>->] (lmc) to node[above]{$S \otimes s' \otimes X$} (lml);
    \draw[>->] (lmc) to node[above]{$s \otimes S' \otimes X$} (lmr);
    \draw[->] (lml) to node[below]{$c_s$} (b);
    \draw[->] (lmr) to node[below]{$c_{s'}$} (b);
  \end{pic}\]
  The upper triangle and central squares commute trivially. The lower quadrilateral commutes and equals $c_{s \otimes s'}$ because $s \otimes s' \in U$ and $c$ is a cocone. Hence the outer diagram commutes, showing $c_r = c_{r'} \circ g$ as required. 
  In particular, taking $R' = R$ shows that $c_r$ is independent of the choice of $s$.
\end{proof}

\begin{lemma}\label{lem:cocone}
  Let $\cat{C}$ and $\cat{D}$ be stiff categories, $U \subseteq \ISub(\cat{C})$ be idempotent, and $c_s \colon S \otimes X \to C$ be a cocone over $D(U,X)$. If a functor $F \colon \cat{C} \to \cat{D}$ preserves monomorphisms of the form $s \otimes X \rightarrowtail X$, for subunits $s$, and the pullbacks~\eqref{eq:stiff-pullback}, then $F(c_s)$ is a cocone over $D\big( F(U), F(X) \big)$, where $F(U) = \{F(s) \mid s \in U\}$.
\end{lemma}
\begin{proof}
  Clearly, if $s \otimes X \leq t \otimes X$ then $F(s \otimes X) \leq F(t \otimes X)$, and $F(c_s)$ respects the inclusion. 
  Conversely, suppose that $F(s \otimes X) \leq F(t \otimes X)$ via some morphism $f$, and consider the following diagram.
  \[\begin{tikzpicture}[xscale=5,yscale=1.5]
    \node (tl) at (0,1) {$F(S \otimes T \otimes X)$};
    \node (tr) at (1,1) {$F(T \otimes X)$};
    \node (bl) at (0,0) {$F(S \otimes X)$};
    \node (br) at (1,0) {$F(C)$};
    \node (x) at (1.4,-.7) {$F(X)$};
    \draw[>->] (tl) to node[above]{$F(s \otimes T \otimes X)$} (tr);
    \draw[>->] (tl) to node[left]{$F(S \otimes t \otimes X)$} (bl);
    \draw[->] (tr) to node[right]{$F(c_t)$} (br);
    \draw[->] (bl) to node[below]{$F(c_s)$} (br);
    \draw[>->] (tr) to[out=-30,in=90,looseness=.7] node[right]{$F(t \otimes X)$} (x);
    \draw[>->] (bl) to[out=-60,in=180] node[below]{$F(s \otimes X)$} (x);
    \draw[->](bl) to node[above]{$f$} (tr);
  \end{tikzpicture}\]
  The outer rectangle commutes by bifunctoriality, and $F(t \otimes X) \circ f = F(s \otimes X)$ by assumption. 
  Hence the upper left triangle commutes because $F(t \otimes X)$ is monic by stiffness and the assumption on $F$. The inner square commutes and is equal to $F(c_{s \otimes t})$ by definition of $D(U,X)$. Since the outer rectangle is a pullback, the leftmost vertical morphism is invertible and hence $F(c_t) \circ f = F(c_s)$. 
\end{proof}

Now suppose a diagram $D(U,X)$ has a colimit $c_s^X \colon S \otimes X \to \colim D(U,X)$ for each idempotent $U \subseteq \ISub(\cat{C})$ and object $X$.
Then there are two canonical morphisms. 
First, a mediating map $\colim D(U,I) \to I$ to the cocone $s \colon S \to I$.
\begin{equation}\label{eq:sU}\begin{pic}[xscale=2,yscale=1.5]
  \node (c) at (2,0) {$\colim D(U,I)$};
  \node (i) at (2,-1) {$I$};
  \node (s) at (0,0) {$S$};
  \draw[->,dashed] (c) to (i);
  \draw[->] (s) to node[below]{$s$} (i);
  \draw[->] (s) to node[above]{$c_s^I$} (c);
\end{pic}\end{equation}
Second, in a stiff category it follows from applying Lemma~\ref{lem:cocone} to $(-) \otimes X$ that there is a unique map making the following triangle commute for all $s \in U$:
\begin{equation}\label{eq:deltaUX}\begin{pic}[xscale=2,yscale=1.5]
    \node (t) at (0,0) {$S \otimes X$};
    \node (br) at (2,-1) {$(\colim D(U,I)) \otimes X$};
    \node (bl) at (2,0) {$\colim D(U,X)$};
    \draw[->,dashed] (bl) to (br);
    \draw[->] (t) to node[above]{$c_s^X$} (bl);
    \draw[->] (t) to node[below]{$c_s^I \otimes X$} (br);
\end{pic}\end{equation}

If $\cat{C}$ has universal joins of $U$ then $\bigvee U = \colim D(U,I)$ and~\eqref{eq:sU} is monic, and~\eqref{eq:deltaUX} is invertible by definition.
We now set out to prove the converse.

\begin{lemma}\label{lem:spatialcharacterisation:monic}
  Let $\cat{C}$ be a stiff category, and let $U \subseteq \ISub(\cat{C})$ be idempotent.
  Suppose that $D(U,X)$ has a colimit for each object $X$ and that each morphism \eqref{eq:deltaUX} is an isomorphism.
  If the morphism $\colim D(U,I) \to I$ of~\eqref{eq:sU} is monic, then it is a subunit.
\end{lemma}
\begin{proof}
  Write $s_U$ for this morphism, which is monic by assumption.
  For each $s \in U$, we claim $S \otimes s_U \colon S \otimes \colim D(U,I) \to S$ is an isomorphism. 
  It is monic because
  \[
    s_U \circ c_s \circ (S \otimes s_U)
    = s \otimes s_U
    = s_U \circ \big(s \otimes \colim D(U,I)\big)
  \]
  where $s_U$ and $s \otimes \colim D(U,I)$ are monic by stiffness. But it is also split epic since $(S \otimes s_U) \circ (S \otimes c_s) = S \otimes s$ is an isomorphism.

  Now since $s \circ (S \otimes s_U) = s_U \circ (s \otimes \colim D(U,I))$,  bifunctoriality of $\otimes$ shows that for all $s, t \in U$:
  \[
    s \otimes \colim D(U,I) \leq t \otimes \colim D(U,I)
    \quad\iff\quad
    s \leq t
  \]
  This gives an isomorphism of diagrams $S \otimes s_U \colon S \otimes \colim D(U,I) \to S$ from $D\big(U, \colim D(U,I)\big)$ to $D(U,I)$. Writing $c_s \colon S \to \colim D(U,I)$ for the latter colimit, $c_s \otimes \colim D(U,I)$ is a colimit for the former by assumption. Hence the unique map making the following square commute
  \[\begin{pic}[xscale=4,yscale=1.5]
    \node (tl) at (0,1) {$S \otimes \colim D(U,I)$};
    \node (tr) at (1,1) {$S$};
    \node (bl) at (0,0) {$\colim D(U,I) \otimes \colim D(U,I)$};
    \node (br) at (1,0) {$\colim D(U,I)$};
    \draw[->,dashed] (bl) to (br);
    \draw[->] (tl) to node[above]{$S \otimes s_U$} (tr);
    \draw[->] (tl) to node[left]{$c_s \otimes \colim D(U,I)$} (bl);
    \draw[->] (tr) to node[right]{$c_s$} (br);
  \end{pic}\]
  is invertible. But this map is just $\colim D(U,I) \otimes s_U$, so $s_U$ is a subunit.
\end{proof}

We can now characterise {locale-based} categories purely categorically.

\begin{theorem} \label{thm:spatialcharacterisation}
  A stiff category $\cat{C}$ has universal (finite, directed) joins if and only if for each idempotent (and finitely bounded, directed) $U \subseteq \ISub(\cat{C})$:
  \begin{itemize}
    \item the diagram $D(U,X)$ has a colimit;
    \item the canonical morphism~\eqref{eq:sU} is monic;
    \item the canonical morphism~\eqref{eq:deltaUX} is invertible.
  \end{itemize}
\end{theorem}
\begin{proof}
  The conditions are clearly necessary, as already discussed. Conversely, suppose that they hold and let $U \subseteq \ISub(\cat{C})$ be as above. Lemma~\ref{lem:idempotentsetsofidempotentsubunits} lets us assume $U = \downset U$.  
  Then $s_U \colon \colim D(U,I) \to I$ is a subunit by Lemma~\ref{lem:spatialcharacterisation:monic}, and by definition $s \leq s_U$ for all $s \in U$. 
  Now suppose that $t$ is also an upper bound in $\ISub(\cat{C})$ of all $s \in U$.
  Then the inclusions $i_{s,t} \colon S \to T$ form a cocone over $D(U,I)$.
  Hence there is a unique mediating map $f \colon \colim D(U,I) \to T$ with $i_{s,t} = f \circ c_s^I$ for all $s \in U$.
  But then
  \[
    t \circ f \circ c_s^I = t \circ i_{s,t} = s = s_U \circ c_s^I
  \]
  for all $s \in U$.
  Because the $c_s^I$ are jointly epic, $t \circ f = s_U$, so that $s_U \leq t$. 
  Therefore indeed $\colim D(U,I) = \bigvee U$. 
  Thus universal finite or directed joins follow by Proposition~\ref{prop:unify}, and so arbitrary ones by Corollary~\ref{cor:spatialiffboth}.
\end{proof}

\section{Completions} 
\label{sec:broad}

Our goal for this section is to embed a stiff category $\cat{C}$ into one with any given kind of universal joins of subunits, including a {locale-based} category. One might think to work with the free cocompletion of $\cat{C}$, the category of presheaves $\Psh{\cat{C}}=[\cat{C}\op, \cat{Set}]$. Here, $\Psh{\cat{C}}$ is endowed with the Day convolution $\ootimes$ as tensor; for details see \appref. Although $\Psh{\cat{C}}$ has a complete lattice of subunits, we will see that it has two problems: it is in general not firm, and it has too many subunits to be the {locale-based} completion. We will remedy both problems by passing to a full subcategory of so-called broad presheaves.

First, note that any subunit $s$ in a firm category $\cat{C}$ induces a subunit ${s \circ (-)} \colon \cat{C}(-,S) \to \cat{C}(-,I)$ in $\Psh{\cat{C}}$ since the Yoneda embedding is monoidal, full, and faithful, and preserves all limits and hence monomorphisms. 

\begin{proposition}\label{prop:colimitspresheaves}
  If $\cat{C}$ is a cocomplete regular category, and  for all objects $A$ the functors $A \otimes (-)$ preserve colimits, then $\ISub(\cat{C})$ is a complete lattice.
  Thus, if $\cat{C}$ is any braided monoidal category, then $\ISub(\Psh{\cat{C}})$ is a complete lattice.
\end{proposition}
\begin{proof}
  In cocomplete regular categories, the subobjects of a fixed object form a complete lattice~\cite[Proposition~4.2.6]{borceux:1}. Explicitly, let $s_i \colon S_i \rightarrowtail I$ be a family of subunits.
  Choose a coproduct $c_i \colon S_i \to C$.
  The unique mediating map $C \to I$ factors through a monomorphism $\bigvee s_i \colon S \rightarrowtail I$, which is the supremum.
  \[\begin{tikzpicture}[xscale=2,yscale=1]
    \node (Si) at (0,2.5) {$S_i$};
    \node (Sj) at (2,2.5) {$S_j$};
    \node (C) at (1,2) {$C$};
    \node (S) at (1,1) {$S$};
    \node (I) at (1,0) {$I$};
    \draw[->] (Si) to node[below]{$c_i$} (C);
    \draw[->] (Sj) to node[below]{$c_j$} (C);
    \draw[>->] (Si) to[out=-90,in=150] node[left]{$s_i$} (I);
    \draw[>->] (Sj) to[out=-90,in=30] node[right]{$s_j$} (I);
    \draw[->>,dashed] (C) to node[right]{$e$} (S);
    \draw[>->,dashed] (S) to node[right]{$s$} (I);
  \end{tikzpicture}\]

  Next we show that $\bigvee s_i$ is a subunit. Let $c = s \circ e \colon C \to I$. We claim that
  \[\begin{pic}[xscale=5]
    \node (tl) at (0,1) {$C \otimes C$};
    \node (tr) at (1,1) {$C$};
    \node (bl) at (.5,0) {$\coprod_i S_i \otimes C$};
    \draw[->] (tl) to node[above]{$C \otimes c$} (tr);
    \draw[->] (tl) to node[left=3mm]{$\simeq$} (bl);
    \draw[->] (bl) to node[right, pos=0.1]{$\ \ \ \ \ \coprod_i (S_i \otimes c)$} (tr);
  \end{pic}\]
  is a regular epimorphism. Since colimits commute with colimits, it suffices to check that each $S_i \otimes c$ is a regular epimorphism. But this is so: if $S_i \otimes c = m \circ f$ for some regular epimorphism $f$ and monomorphism $m$, then
  $m \circ f \circ (S_i \otimes c_i) = (S_i \otimes c) \circ (S_i \otimes c_i) = S_i \otimes s_i$ is an isomorphism by idempotence of $s_i$, so that $m$ is split epic as well as monic and hence an isomorphism. 

  Now the topmost two rectangles in the following diagram commute.
  \[\begin{tikzpicture}[xscale=4,yscale=1.25]
    \node (Si) at (.5,2.5) {$S_i$};
    \node (C) at (1,2) {$C$};
    \node (S) at (1,1) {$S$};
    \node (I) at (1,0) {$I$};
    \node (II) at (2,0) {$I \otimes I$};
    \node (SS) at (2,1) {$S \otimes S$};
    \node (CC) at (2,2) {$C \otimes C$};
    \node (SiSi) at (2.5,2.5) {$S_i \otimes S_i$};
    \draw[->] (SiSi) to node[above]{$S_i \otimes s_i$} (Si);
    \draw[->>] (CC) to node[above]{$C \otimes c$} (C);
    \draw[->] (SS) to node[above]{$\lambda_S \circ (S \otimes s)$} (S);
    \draw[->] (II) to node[below]{$\lambda_I$} (I);
    \draw[>->] (S) to node[left]{$s$} (I);
    \draw[->>] (C) to node[left]{$e$} (S);
    \draw[>->] (SS) to node[right]{$s \otimes s$} (II);
    \draw[->>] (CC) to node[right]{$e \otimes e$} (SS);
    \draw[->] (Si) to node[below]{$c_i$} (C);
    \draw[->] (SiSi) to node[right=2mm,pos=.9]{$c_i \otimes c_i$} (CC);
    \draw[->] (SiSi) to[out=-90,in=0,looseness=.4] node[right]{$s_i \otimes s_i$} (II);
    \draw[->] (Si) to[out=-90,in=180,looseness=.5] node[left]{$s_i$} (I);
  \end{tikzpicture}\]
  The left and right triangles commute by construction, and the bottom rectangle commutes by bifunctoriality of the tensor and naturality of $\lambda$. Because $e$ is a coequaliser, so are $C \otimes e$ and $e \otimes S$, and hence so is $e \otimes e$. Therefore both vertical morphisms factor as regular epimorphisms followed by monomorphisms, and the mediating morphism, which must be $\lambda_S \circ (S \otimes s)$ by uniqueness, is an isomorphism. 
  Thus $S \otimes s$ is an isomorphism, as required.

  The second statement now follows, because $\Psh{\cat{C}}$ is regular and cocomplete, and the functors $F \ootimes (-)$ are cocontinuous~\cite{imkelly:day}.
\end{proof}

However, the subunits in $\Psh{\cat{C}}$ are in general not well behaved. 

{
\begin{proposition} \label{prop:PshFirmCounter}
  Consider the commutative monoid $M=[0,1) \times [0,\infty)$ under 
  \[
    (a,b) + (c,d) = 
    \begin{cases} 
      (a+c, b+d) & \text{ if } a + c < 1 \\ 
      (a + c - 1, b + d + 1) & \text{ if } a + c \geq 1
    \end{cases}
  \]
  with unit $(0,0)$. Then $M$ is a firm one-object category, but $\Psh{M}$ is not firm.
\end{proposition}
}
\begin{proof}
  The identity $(0,0)$ represents the only subunit of the one-object category $M$, which is therefore firm.
  \appref proves that $\Psh{M}$ is not firm.
\end{proof}

Moreover, $\Psh{\cat{C}}$ may have subunits that are not suprema of subunits of $\cat{C}$.

{
\begin{proposition}\label{prop:daycounterexample}
  In general $\ISub(\Psh{\cat{C}})$ is not the free frame on $\ISub(\cat{C})$.
\end{proposition}
}
\begin{proof}
  Consider a commutative unital quantale $Q$ as a firm category. By their description in \appref, any subunit in $\widehat{Q}$ is given by a suitable downward closed subset $S \subseteq\downset \qunit \subseteq Q$ such that $\forall x \in S\, \exists y,z \in S \colon x \leq yz$, and to be a subunit it suffices for $S$ to be directed. 

  In particular, take $Q=[0,\infty]$ under the opposite order and addition.
  Then $\ISub(Q) = \{ 0, \infty \}$, whose free completion to a frame is its collection of downsets
  $
    \big\{ \emptyset, \{\infty\}, \{0,\infty\} \big\}
  $.
  However, by the above description of subunits in $\widehat{Q}$ it is easy to see that 
  $
    \ISub(\widehat{Q}) \supseteq \big\{ \emptyset, \{\infty\}, [0,\infty], (0,\infty] \big\}
  $.
\end{proof}

Instead, to complete $\ISub(\cat{C})$ to a distributive lattice, preframe, or frame, we will consider certain full subcategories of $\Psh{\cat{C}}$. 

\begin{definition}\label{def:broad}
  A presheaf on a braided monoidal category $\cat{C}$ is \emph{(finitely, directedly) broad} when it is naturally isomorphic to one of the form
  \[
    \broad{U}{X} \colon A \mapsto \{ f \colon A \to X \mid \text{$f$ restricts to some $s \in U$} \}
  \]
  for a (finitely bounded, directed) family $U$ of subunits and an object $X$.

  Write $\Spsh{\cat{C}}$ ($\FUsh{\cat{C}}$, $\DUsh{\cat{C}}$) for the full subcategory of (finitely, directedly) broad presheaves.
  We will also write $\widehat{U}$ for $\broad{U}{I}$, and $\widehat{X}$ for $\broad{\{1\}}{X}$.
\end{definition}

We will see below that the broad presheaves are precisely the colimits of the diagrams $D(\{\hat{s} \mid s \in U\}, \hat{X})$, and leave open the possibility of characterising when a given presheaf is broad in terms not referring to $U$ or $X$.

The following lemma shows that broad presheaves are closed under (Day) tensor products and so form a monoidal category.

\begin{lemma}\label{lem:broadtensor}
  For any objects $X$, $Y$ and families of subunits $U$, $V$ in a stiff category $\cat{C}$, there is a (unique) natural isomorphism making
  \begin{equation} \label{eq:day-useful}
   \begin{pic}[xscale=4,yscale=1.5]
    \node (tl) at (0,1) {$\broad{U}{X} \ootimes \broad{V}{Y}$};
    \node (tr) at (0,0) {$\widehat{X} \ootimes \widehat{Y}$};
    \node (bl) at (1,1) {$\broad{U \otimes V}{X \otimes Y}$};
    \node (br) at (1,0) {$\widehat{X \otimes Y}$};
    \draw[->] (tl) to node[left]{$u \ootimes v$} (tr);
    \draw[->,dashed] (tl) to node[above]{$\simeq$} (bl);
    \draw[->] (tr) to node[above]{$\simeq$} (br);
    \draw[>->] (bl) to (br);
   \end{pic}  
  \end{equation}
  commute, where $U \otimes V = \{ s \wedge t \mid s \in U, t \in V \}$, and $u, v$ are the inclusions. 
\end{lemma}
\begin{proof}
  See \appref. 
\end{proof}

We now describe the subunits in each completion.

\begin{proposition}\label{prop:broadidempotent}
  If $\cat{C}$ is stiff, the subunits in $\Spsh{\cat{C}}$ ($\FUsh{\cat{C}}$, $\DUsh{\cat{C}}$) are the presheaves of the form $\widehat{U}$ for (finitely bounded, directed) $U \subseteq \ISub(\cat{C})$. 
\end{proposition}
\begin{proof}
  Clearly $\widehat{U}$ is a subunit. Conversely, if $\eta \colon \broad{U}{X} \to \widehat{I}$ is a subunit then
  \[
    s_X = \eta_{S \otimes X}(s \otimes X) \colon S \otimes X \to I
  \]
  will be proven to be a subunit in $\cat{C}$ for each $s \in U$.

  Given this, let $U' = \{ s_X \mid s \in U \}$, noting that $\widehat{U'}$ again belongs to each respective category, and consider the function $\broad{U}{X}(A) \to \broad{U'}{I}(A)$ given by $((s \otimes X) \circ f) \mapsto s_X \circ f$. It is surjective by definition of $U'$, clearly natural, and is well-defined and injective since 
  \begin{align*}
    s_X \circ f = s'_X \circ f'
    & \iff \eta(s \otimes X) \circ f = \eta(s' \otimes X) \circ f' \\
    & \iff \eta((s \otimes X) \circ f)) = \eta((s' \otimes X) \circ f') \\
    & \iff (s \otimes X) \circ f =(s' \otimes X) \circ f'
  \end{align*}
  by naturality and injectivity of $\eta$. 

  Let us show that $s_X$ is indeed a subunit. By stiffness of $\cat{C}$ each morphism $(s \otimes X)$ is monic, and so by the above argument $s_X$ is, too.

  Next we show $s_X \otimes S \otimes X$ is invertible. 
  Notice that $\broad{U}{X} = \broad{\downset U}{X}$, so we may assume that $U$ is idempotent. The fact that $\eta$ is a subunit means precisely that each map
  \begin{align}\label{eq:idempotentpresheafbijection}\tag{$*$}
    \broad{U}{X \otimes X}(A) & \to \broad{U}{X}(A) \notag\\
    (s \otimes (X \otimes X)) \circ f & \mapsto (s_X \otimes X) \circ f
  \end{align}
  is a well-defined bijection, where $f \colon A \to S \otimes X \otimes X$ and $s \in U$.

  Now note that $S \otimes s_X \otimes X$ is monic, since by injectivity of~\eqref{eq:idempotentpresheafbijection}, $s_X \otimes X$ is monic, and it is easy to see from stiffness that for any subunit $s$ and monic $m$ that $S \otimes m$ is again monic. 
  Moreover it is split epic and hence an isomorphism, since by surjectivity of~\eqref{eq:idempotentpresheafbijection} there is some $f$ with $(s_X \otimes X) \circ f = s \otimes X$, and $S \otimes (s \otimes X)$ is always split epic by idempotence of $s$. 
\end{proof}

For any semilattice, as well as its downsets forming its free completion to a frame, recall that its free completion to a preframe is given by its collection of directed downsets~\cite[Theorem~9.1.5]{vickers:topology}; and that its free completion to a distributive lattice is given by its finitely bounded downsets~\cite[I.4.8]{johnstone:stonespaces}, with (directed, finite) joins given by unions. 

\begin{corollary}
  The subunits in $\FUsh{\cat{C}}$, $\DUsh{\cat{C}}$, and $\Spsh{\cat{C}}$, are the free completion of $\ISub(\cat{C})$ to a distributive lattice, preframe, and frame, respectively.
\end{corollary}
\begin{proof}
  For any $U, V \subseteq \ISub(\cat{C})$ it is easy to see that $\widehat{U} \leq \widehat{V} \iff U \leq {\downset V}$. In particular $\widehat{U} = \widehat{\downset\! U}$ as we have already noted. Hence by Proposition~\ref{prop:broadidempotent}, subunits in each category correspond to the respective kinds of downset $U \subseteq \ISub(\cat{C})$.
\end{proof}

Next let us note that each of our constructions are again stiff.

\begin{lemma}\label{lem:broadpresheavesstiff}
  If a monoidal category $\cat{C}$ is stiff, then so are $\DUsh{\cat{C}}$, $\FUsh{\cat{C}}$ and $\Spsh{\cat{C}}$.
\end{lemma}
\begin{proof}
  For any object $\broad{U}{X}$ and subunit $V \colon \widehat{V} \to \widehat{I}$ in $\Spsh{\cat{C}}$ we need to show that the morphism $\broad{U}{X} \otimes V$ is monic. This holds since the obvious morphism $\broad{U}{X} \otimes \widehat{V} \to \widehat{X}$ factors over it, and is itself monic by equation~\eqref{eq:day-useful} of Lemma~\ref{lem:broadtensor}. 

  By the same result, for the pullback property we must show each diagram
  \[
  \begin{pic}[xscale=4,yscale=1.5]
    \node (tl) at (0,1) {$\broad{U \otimes V \otimes W}{X}$};
    \node (tr) at (1,1) {$\broad{U \otimes W}{X}$};
    \node (bl) at (0,0) {$\broad{V \otimes W}{X}$};
    \node (br) at (1,0) {$\broad{W}{X}$};
    \draw[>->] (tl) to node[above]{} (tr);
    \draw[>->] (tl) to node[left]{} (bl);
    \draw[>->] (tr) to node[right]{} (br);
    \draw[>->] (bl) to node[below]{} (br);
    \draw (.1,.7) to (.15,.7) to (.15,.825);
  \end{pic}
  \]
 to be a pullback in $\Spsh{{\cat{C}}}$. For this it suffices to check that applying the diagram to each object $A$ yields a pullback in $\cat{Set}$, or equivalently that any morphism $f \colon A \to X$ factoring over $u \otimes w \otimes X$ and $v \otimes w' \otimes X$ for some $u \in U, v \in V$ and $w, w' \in W$ factors over $u' \otimes v' \otimes w'' \otimes X$ for some $u' \in U, v' \in V, w'' \in W$. But this follows easily from the pullbacks~\eqref{eq:stiff-pullback} taking $u'=u$, $v'=v$ and $w'' = w \wedge w'$, again for convenience assuming $W$ to be idempotent.
\end{proof}

The next lemma shows that $\Spsh{\cat{C}}$ formally adds to $\cat{C}$ the colimits of the diagrams $D(U,X)$ for all suitable $U \subseteq \ISub(\cat{C})$ and objects $X$.

\begin{lemma}\label{lem:morphismsofbroadpresheaves}
  Let $\cat{C}$ be firm, and let $U, V \subseteq \ISub(\cat{C})$ be idempotent. 
  Morphisms $\alpha \colon \broad{U}{X} \to \broad{V}{Y}$ of broad presheaves correspond to cocones $c_s \colon S \otimes X \to Y$ over $D(U, X)$ for which each $c_s$ restricts to some $t \in V$.
\end{lemma}
\begin{proof}
  Given $\alpha$ and $s \in U$, by naturality we may define such a cocone by $c_s = \alpha_{S \otimes X}(s \otimes X)$. Conversely, given a cocone as above define
  \[
    \alpha_A\big((s \otimes X) \circ g\big) = c_s \circ g
  \]
  for each $g \colon A \to S \otimes X$. This is clearly natural and is well-defined; indeed if $(s \otimes X) \circ g = (t \otimes X) \circ h$ then since~\eqref{eq:stiff-pullback} is a pullback this morphism factors as $(s \otimes t \otimes X) \circ k$ for some $k$, then with $c_s \circ g =c_{s \wedge t} \circ k = c_t \circ h$ since the $(c_s)$ form a cocone. Clearly these two assignments are inverses.
\end{proof}

Finally we can prove that our free constructions have the desired properties.

\begin{theorem}\label{thm:spatialcompletion}
  If $\cat{C}$ is a stiff category, then:
  \begin{itemize}
    \item $\FUsh{\cat{C}}$ has universal finite joins of subunits;
    \item $\DUsh{\cat{C}}$ has universal directed joins of subunits;
    \item $\Spsh{\cat{C}}$ is {locale-based}. 
  \end{itemize}
\end{theorem}
\begin{proof}
  Consider the final statement first.
  Lemma~\ref{lem:broadpresheavesstiff} makes $\Spsh{\cat{C}}$ stiff.
  Let $\mathcal{U}$ be an idempotent family of subunits in $\Spsh{\cat{C}}$.
  By Proposition~\ref{prop:broadidempotent}, its elements are of the form $\widehat{U}$ for some $U \subseteq \ISub(\cat{C})$. 
  Also, its supremum in $\ISub(\Spsh{\cat{C}})$ is given by $\broad{\bigcup \mathcal{U}}{I}$ where we write $\bigcup \mathcal{U} = \bigcup \{ U \mid \widehat{U} \in \mathcal{U} \}$.

  Let $V \subseteq \ISub(\cat{C})$, and let $Y$ be an object in $\cat{C}$. 
  We have to prove that the inclusions $\widehat{U} \ootimes \broad{V}{Y} \to \bigcup \mathcal{U} \ootimes \broad{V}{Y}$ are a colimit of the diagram $D(\mathcal{U},\broad{V}{Y})$ in $\Spsh{\cat{C}}$.
  By Lemma~\ref{lem:broadtensor}, we may equivalently consider the inclusions
  \[
    \broad{U \otimes V}{Y} \hookrightarrow \broad{
 (\bigcup \mathcal{U}) \otimes V }{Y}\text.
  \]
  These certainly form a cocone. The questions is whether it is a universal one.
  Suppose that $\alpha_U \colon \broad{U \otimes V}{Y} \to \broad{W}{Z}$ is another cocone.
  Define a natural transformation $\beta \colon \broad{ (\bigcup \mathcal{U}) \otimes V}{Y} \to \broad{W}{Z}$ by $\beta_A(f) = (\alpha_U)_A(f)$ for any $f \colon A \to X$ that restricts to $U \in \mathcal{U}$. 

  Now $\beta$ is indeed well-defined, since if $f$ also restricts to $U' \in \mathcal{U}$ then by the pullback~\eqref{eq:stiff-pullback}, it also restricts to $U \cap U' \in \mathcal{U}$, so that $(\alpha_U)_A(f) = (\alpha_{U \cap U'})_A(f) = (\alpha_{U'})_A(f)$. By definition $\beta$ is the unique natural transformation making the following triangle commute:
  \[\begin{pic}[xscale=6,yscale=1.5]
    \node (tl) at (0,0.7) {$\broad{U\otimes V}{Y}$};
    \node (tr) at (0.7,0.7) {$\broad{(\bigcup \mathcal{U}) \otimes V}{Y}$};
    \node (br) at (0.7,0) {$\broad{W}{Z}$};
    \draw[>->] (tl) to (tr);
    \draw[->] (tr) to node[right]{$\beta$} (br);
    \draw[->] (tl) to node[below]{$\alpha_U$} (br);
  \end{pic}\]
  Hence the inclusions indeed form a colimit, and $\Spsh{\cat{C}}$ is {locale-based}.
  The proofs of the first two statements are identical, observing that if $U, V \subseteq \ISub(\cat{C})$ and $\mathcal{U} \subseteq \ISub(\FUsh{\cat{C}})$ or $\ISub(\DUsh{\cat{C}})$ are finitely bounded or directed, then so are $U \otimes V$ and $\bigcup \mathcal{U}$.
\end{proof}

We end this section by showing that the {locale-based} completion cannot be read in the traditional topological sense, in that broad presheaves are not sheaves for any Grothendieck topology.

\begin{proposition}
  There is a firm category $\cat{C}$ for which there is no Grothendieck topology $J$ with $\Spsh{\cat{C}} \simeq \mathrm{Sh}(\cat{C},J)$.
\end{proposition}
\begin{proof}
  Suppose that $\Spsh{\cat{C}}$ is a Grothendieck topos. 
  Then it is a reflective subcategory of $\Psh{\cat{C}}$~\cite[Proposition~3.5.4]{borceux:3}.
  Hence $\Spsh{\cat{C}}$ has a terminal object $\broad{U}{X}$ that, because right adjoints preserve limits, must equal the terminal object of $\Psh{\cat{C}}$.
  Therefore, for all objects $A$ of $\cat{C}$, the set $\broad{U}{X}(A)$ must be a singleton.
  This means that for all objects $A$, there is a unique morphism $A \to X$ that restricts to some $s \in U$.

  Suppose $\ISub(\cat{C})=\{I\}$. Since every morphism restricts to $I$, now $X$ must be a terminal object.
  But there exists a braided monoidal category $\cat{C}$ with only one subunit but no terminal object: any nontrivial abelian group.
\end{proof}

\begin{remark}
  In future it would be natural to consider the above completions with presheaves valued in a category other than $\cat{Set}$~\cite{borceuxquinteiro:sheaves}.
  After all, {Proposition}~\ref{ex:quantale} is enriched over complete lattices, {Proposition}~\ref{prop:modules} is enriched over abelian groups, and {Proposition}~\ref{prop:hilbertmodules} is enriched over normed vector spaces. 
  Proposition~\ref{prop:colimitspresheaves} holds for enriching categories $\cat{V}$ that are complete, cocomplete, locally small, and symmetric monoidal closed~\cite{imkelly:day}, covering all these examples.
  But an enriched version of Definition~\ref{def:broad} would require taking the subobject of $[A,X]$ in $\cat{V}$ that restricts to some $s \in U$. 
\end{remark}

\section{Universality of the completions}
\label{sec:completion}

Finally, let us prove that the {locale-based} completion $\Spsh{\cat{C}}$ and our other constructions $\FUsh{\cat{C}}$ and $\DUsh{\cat{C}}$ indeed have universal properties. 

\begin{definition}
  A \emph{morphism} of categories with universal (finite, directed) joins of subunits is a braided monoidal functor $F \colon \cat{C} \to \cat{D}$ that preserves subunits and their (finite, directed) suprema. For short we call morphisms of categories with universal joins of subunits simply morphisms of {locale-based} categories. 
\end{definition}

Here, a functor $F$ is monoidal when it comes equipped with coherent isomorphisms $\varphi_{A,B} \colon F(A) \otimes F(B) \to F(A \otimes B)$ and $\varphi \colon I \to F(I)$; these need to be invertible to make sense of preservation of subunits: if $s \in \ISub(\cat{C})$, then $\varphi^{-1} \circ F(s) \in \ISub(\cat{D})$. 

By Lemma~\ref{lem:cocone} and Theorem~\ref{thm:spatialcharacterisation}, a morphism is equivalently a braided monoidal functor $F \colon \cat{C} \to \cat{D}$ with $F\big(\colim D(U,X)\big) = \colim D\big(F(U),F(X)\big)$ for (finitely bounded, directed) idempotent $U \subseteq \ISub(\cat{C})$ and objects $X$ of $\cat{C}$.

\begin{definition}
  The \emph{{locale-based} completion} of a braided monoidal category $\cat{C}$ is a monoidal functor $y \colon \cat{C} \to \cat{D}$ that preserves subunits such that $\cat{D}$ is {locale-based}, and any monoidal functor $\cat{C} \to \cat{E}$ into a {locale-based} category that preserves subunits factors as $y$ followed by a morphism of {locale-based} categories $G$ that is unique up to a unique monoidal natural isomorphism $\gamma$ with $\gamma_y = \id[G]$.
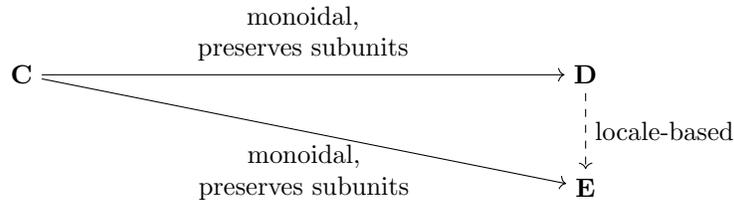
  \[\begin{pic}[xscale=15,yscale=1.5]
      \node (tl) at (0,1) {$\cat{C}$};
      \node (tr) at (0.5,1) {$\cat{D}$};
      \node (bl) at (.5,0) {$\cat{E}$};
      \draw[->] (tl) to node[above]{\begin{tabular}{c}monoidal,\\preserves subunits\end{tabular}} (tr);
      \draw[->,dashed] (tr) to node[right]{{locale-based}} (bl);
      \draw[->] (tl) to node[below]{\begin{tabular}{c}monoidal,\\preserves subunits\end{tabular}} (bl);
  \end{pic}\]
  A \emph{completion under universal finite} or \emph{directed joins of subunits} of $\catC$ is defined similarly.
\end{definition}

\begin{theorem}\label{thm:spatialcompletionuniversal}
  If $\cat{C}$ is a stiff category, then via the Yoneda embedding its 
  \begin{itemize}
    \item completion under universal finite joins of subunits is $\FUsh{\cat{C}}$;
    \item completion under universal directed joins of subunits is $\DUsh{\cat{C}}$;
    \item {locale-based} completion is $\Spsh{\cat{C}}$.
  \end{itemize}
\end{theorem}
\begin{proof}
  We prove the {locale-based} case, the others being identical.
  For any monoidal functor $F \colon \cat{C} \to \cat{D}$ into a {locale-based} category, we need to show that there is a morphism $\overline{F} \colon \Spsh{\cat{C}} \to \cat{D}$ with $\overline{F} \circ y = F$, where $y$ is the Yoneda embedding. 

  Because $\broad{U}{X} = \broad{\downset U}{X}$ for any $U \subseteq \ISub(\cat{C})$, we may assume that $U$ is idempotent. 
  Because $F$ is monoidal, $F(U)$ is idempotent too.
  On objects, the requirement $\overline{F} \circ y = F$ forces us to define 
  \begin{align*}
    \overline{F}\broad{U}{X}
    &=\overline{F}\big(\colim D(y(U),\widehat{X})\big) \\
    &=\colim D\big(\overline{F} \circ y(U), \overline{F} \circ y(X)\big) \\
    &=\colim D\big(F(U), F(X)\big) \\
    &\simeq \big(\bigvee F(U)\big) \otimes F(X)\text.
  \end{align*}

  Now consider morphisms of (broad) presheaves. Any $\alpha \colon \broad{U}{X} \to \broad{V}{Y}$ induces a cocone $\alpha_s = \alpha_{S \otimes X}(s \otimes X) \colon S \otimes X \to Y$ over $D(U,X)$, where, as in Lemma~\ref{lem:morphismsofbroadpresheaves}, each such map factors through $t \otimes Y$ for some $t \in V$. 
  Hence $F(\alpha_s)$ factors through $F(t) \otimes F(Y)$ and hence $\colim D\big(F(V),F(Y)\big) = \overline{F}\broad{V}{Y}$, giving a morphism $\beta_s$ as below. 
  \[\begin{pic}[xscale=5,yscale=1.5]
    \node (tl) at (0,2) {$F(S) \otimes F(X)$};
    \node (tm) at (0.7,2) {$F(S \otimes X)$};
    \node (tr) at (1.5,2) {$F(Y)$};
    
    \node (bl) at (0,0) {$\overline{F}\broad{U}{X}$};
    \node (br) at (1.5,0) {$\overline{F}\broad{V}{Y}$};

    \node (ml) at (0,1) {$\overline{F}\broad{\{s\}}{X}$};
    \node (mr) at (1.5,1) {{$\overline{F}(\widehat{Y})$}};

    \draw[->] (tl) to node[above]{$\simeq$} (tm);
    \draw[->] (tm) to node[above]{$F(\alpha_s)$} (tr);
    \draw[->,dashed] (ml) to node[above]{$\beta_s$} (br);
    \draw[->] (tl) to node[left]{$\simeq$} (ml);
    \draw[>->] (br) to node[left]{} (mr);
    \draw[->,dashed] (bl) to node[below]{$\overline{F}(\alpha)$} (br);

    \draw[-, double distance=.75mm]  (mr) to (tr);
  
    \draw[>->] (ml) to node[above]{$\overline{F}\big(\alpha_s \circ (-)\big)$} (mr);
    \draw[->] (ml) to (bl);
    \draw[>->] (br) to (mr);
  \end{pic}\]
  By Lemma~\ref{lem:cocone}, the upper row forms a cocone over $D\big(F(U),F(X)\big)$ with $s$ ranging over $U$. 
  Because the vertical composite on the right is monic, the $\beta_s$ also form a cocone (after composition with the upper left vertical isomorphism). 
  But $\overline{F}\broad{U}{X}$ is a colimit, so there is a mediating map $\overline{F}(\alpha)$ making the diagram commute. 
  Uniqueness of this map makes $\overline{F}$ functorial. Given our definition of $\overline{F}$ on objects, this assignment $\overline{F}(\alpha)$ is unique with $\overline{F} \circ y = F$, since for each $s \in S$ the lower square commutes by functoriality, with the lower left vertical morphisms forming a colimit. 

  Next, $\overline{F}$ may readily be checked to be (strong) braided monoidal:
  \begin{align*}
    \overline{F}(\broad{U}{X} \ootimes \broad{V}{Y})
    &\simeq \overline{F}\broad{U \otimes V}{X \otimes Y} \\
    &\simeq \big(\bigvee_{s \in U, t \in V} F(s) \wedge F(t) \big) \otimes F(X) \otimes F(Y) \\
    &\simeq \bigvee F(U) \otimes \bigvee F(V) \otimes F(X) \otimes F(Y) \\
    &\simeq \overline{F}{\broad{U}{X}} \otimes \overline{F}{\broad{V}{Y}}
  \end{align*}
  By construction $\overline{F}$ preserves subunits because $\overline{F}\broad{U}{I} = \bigvee F(U)$, as well as their suprema:
  \[
    \overline{F}\big(\bigvee_{U \in \mathcal{U}} \broad{U}{I}\big)
    = \overline{F}\broad{\bigcup \mathcal{U}}{I}
    \simeq \bigvee_{U \in \mathcal{U}} \bigvee_{s \in U} F(s)
    \simeq \bigvee_{U \in \mathcal{U}} \overline{F}\broad{U}{I}
  \]
  Hence $\overline{F}$ is indeed a morphism of {locale-based} categories. 

  Finally, we must show for any other morphism $\overline{F}'$ with $\overline{F}' \circ y = F$ that there is a unique monoidal natural isomorphism $\gamma \colon \overline{F} \to \overline{F'}$ with $\gamma_y = \id[F]$. But this follows from the uniqueness of $\colim D\big(F(U),F(X)\big)$ up to unique isomorphism, and our statement above on the uniqueness of $\overline{F}(\alpha)$. 
\end{proof}

We leave open the question how these completions relate to the free cocompletions in a left exact context in the case of toposes~\cite{menni:thesis}.

Each construction is functorial; we consider the {locale-based} case in detail. Write {$\cat{LocBased}$} for the category of {locale-based} categories and their morphisms, and $\cat{Stiff}$ for the category of stiff categories and braided monoidal functors that preserve subunits. 

\begin{proposition}
  The map $\cat{C} \mapsto \Spsh{\cat{C}}$ defines a functor $\cat{Stiff} \to {\cat{LocBased}}$.
\end{proposition}
\begin{proof}
  For any $F \colon \cat{C} \to \cat{D}$ in $\cat{Stiff}$, define $\Spsh{\cat{C}} \to \Spsh{\cat{D}}$ on objects by $\broad{U}{X} \mapsto \broad{F(U)}{F(X)}$. We have seen that it suffices to consider when $U$ is idempotent. By Lemma~\ref{lem:morphismsofbroadpresheaves}, morphisms $\alpha \colon \broad{U}{X} \to \broad{V}{Y}$ are equivalently cocones over $D(U,X)$ each of whose legs factors over $t \otimes Y$ for some $t \in V$. Map such a cocone $c_s$ to the cocone $F(c_s)$ over $D(F(U),F(X))$. This is well-defined by Lemma~\ref{lem:cocone}, and clearly functorial.
\end{proof}

It follows from Theorem~\ref{thm:spatialcompletionuniversal} that the {locale-based} completion functor of the previous proposition is a left \emph{biadjoint} to the forgetful functor $\cat{{LocBased}} \to \cat{Stiff}$, when we make each category a strict 2-category with 2-cells being monoidal natural transformations (for this it suffices to check that each Yoneda embedding $\catC \to \Spsh{\catC}$ is a \emph{biuniversal arrow}~\cite[Theorem~9.16]{fiore2006pseudo}).

The other constructions $\cat{C} \mapsto \FUsh{\cat{C}}$ and $\cat{C} \mapsto \DUsh{\cat{C}}$ similarly give left biadjoints; write $\cat{UnivFin}$ or $\cat{UnivDir}$ for the category of categories with universal finite or directed joins.

\begin{theorem}
  The following cube of forgetful functors commutes, all functors in the top face have left biadjoints, and the rest have left adjoints. 
  
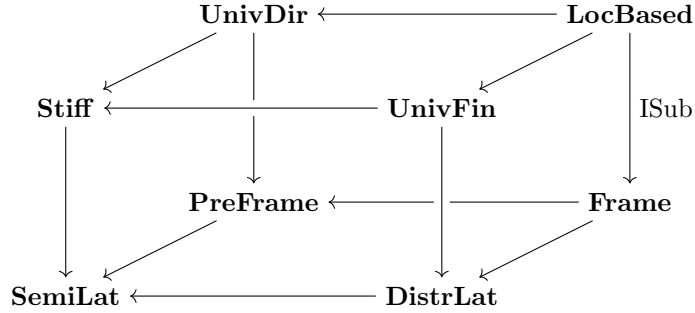
\[
   \begin{pic}[xscale=5,yscale=2.5,cross line/.style={preaction={draw=white, -,line width=6pt}}]
    \node (slat) at (0,0) {$\cat{SemiLat}$};
    \node (pref) at (.5,.5) {$\cat{PreFrame}$};
    \node (fram) at (1.5,.5) {$\cat{Frame}$};
    \node (dist) at (1,0) {$\cat{DistrLat}$};
    \node (stif) at (0,1) {$\cat{Stiff}$};
    \node (udir) at (.5,1.5) {$\cat{UnivDir}$};
    \node (spat) at (1.5,1.5) {$\cat{{LocBased}}$};
    \node (ufin) at (1,1) {$\cat{UnivFin}$};
    \draw[->] (spat) to (ufin);
    \draw[->] (spat) to (udir);
    \draw[->] (udir) to (stif);
    \draw[->] (fram) to (pref);
    \draw[->] (fram) to (dist);
    \draw[->] (pref) to (slat);
    \draw[->] (stif) to (slat);
    \draw[->] (udir) to (pref);
    \draw[->] (dist) to (slat);
    \draw[->] (spat) to node[right]{$\ISub$} (fram);
    \draw[->, cross line] (ufin) to (dist);
    \draw[->, cross line] (ufin) to (stif);
   \end{pic}
  \]
\end{theorem}
\begin{proof}
  All functors in the bottom face have a left adjoint~\cite[Lemma~C1.1.3]{johnstone:elephant}. 
  Explicitly: the free frame on a preframe is given by taking its Scott closed subsets~\cite[Proposition~1]{banaschewski:freeframe},
  and we have already mentioned the free frame, preframe or distributive lattice on a semilattice.
  Observe that all these free constructions take certain types of downward closed subsets.
  Therefore they can be categorified from posets to categories that have universal joins of these types of subsets of subunits.
  The universal property of Theorem~\ref{thm:spatialcompletionuniversal} then holds in each case.
  Hence all functors in the top face of the cube have a left biadjoint.
  Finally, all vertical functors have a left adjoint as in {Proposition}~\ref{ex:semilattice}.
\end{proof}

\bibliographystyle{plain}
\bibliography{tensortopology}

\begin{thebibliography}{10}

\bibitem{adamek1994locally}
J.~Ad{\'a}mek and J.~Rosicky.
\newblock {\em Locally presentable and accessible categories}, volume 189.
\newblock Cambridge University Press, 1994.

\bibitem{balmer:spectrum}
P.~Balmer.
\newblock The spectrum of prime ideals in tensor triangulated categories.
\newblock {\em Journal f{\"u}r die {R}eine und {A}ngewandte {M}athematik},
  588:149--168, 2005.

\bibitem{balmer:tensortriangulargeometry}
P.~Balmer.
\newblock Tensor triangular geometry.
\newblock In {\em International Congress of Mathematicians 2010}, pages
  85--112, 2011.

\bibitem{balmerfavi:telescope}
P.~Balmer and G.~Favi.
\newblock Generalized tensor idempotents and the telescope conjecture.
\newblock {\em Proceedings of the London Mathematical Society},
  102(6):1161--1185, 2011.

\bibitem{balmerkrausestevenson:smashing}
P.~Balmer, H.~Krause, and G.~Stevenson.
\newblock The frame of smashing tensor-ideals.
\newblock {\em Mathematical Proceedings of the Cambridge Philosophical
  Society}, 2018.

\bibitem{banaschewski:freeframe}
B.~Banaschewski.
\newblock Another look at the localic {T}ychonoff theorem.
\newblock {\em Commentationes Mathematicae Universitatis Carolinae},
  29(4):647--656, 1988.

\bibitem{borceux:1}
F.~Borceux.
\newblock {\em Handbook of Categorical Algebra 1: Basic Category Theory}.
\newblock Cambridge University Press, 1994.

\bibitem{borceux:3}
F.~Borceux.
\newblock {\em Handbook of Categorical Algebra 3: Categories of Sheaves}.
\newblock Cambridge University Press, 1994.

\bibitem{borceuxquinteiro:sheaves}
F.~Borceux and C.~Quinteiro.
\newblock A theory of enriched sheaves.
\newblock {\em Cahiers de Topologie et Geometrie Differentielle Categoriques},
  XXXVII-2:145--162, 1996.

\bibitem{boyarchenkodrinfeld:idempotent}
M.~Boyarchenko and V.~Drinfeld.
\newblock Idempotents in monoidal categories.
\newblock http://www.math.uchicago.edu/$\sim$mitya/idempotents.pdf.

\bibitem{boyarchenkodrinfeld:duality}
M.~Boyarchenko and V.~Drinfeld.
\newblock A duality formalism in the spirit of {G}rothendieck and {V}erdier.
\newblock {\em Quantum Topology}, 4(4):447--489, 2013.

\bibitem{boyarchenkodrinfeld:idempotents}
M.~Boyarchenko and V.~Drinfeld.
\newblock Character sheaves on unipotent groups in positive characteristic:
  foundations.
\newblock {\em Selecta Mathematica}, 20(1):125--235, 2014.

\bibitem{brandenburg}
M.~Brandenburg.
\newblock {\em Tensor categorical formulations of algebraic geometry}.
\newblock PhD thesis, University of M{\"u}nster, 2014.

\bibitem{cassidyhebertkelly:factorization}
C.~Cassidy, M.~H{\'e}bert, and G.~M. Kelly.
\newblock Reflective subcategories, localizations and factorization systems.
\newblock {\em Journal of the Australian Mathematical Society}, 38:287--329,
  1985.

\bibitem{clarecrisphigson:adjoint}
P.~Clare, T.~Crisp, and N.~Higson.
\newblock Adjoint functors between categories of {H}ilbert {C*}-modules.
\newblock {\em Journal of the Instititue of Mathematics of Jussieu}, pages
  1--33, 2016.

\bibitem{day:monoidallocalisation}
B.~Day.
\newblock Note on monoidal localisation.
\newblock {\em Bulletin of the Australian Mathematical Society}, 8:1--16, 1973.

\bibitem{day:monoidalmonads}
B.~Day.
\newblock On closed categories of functors {II}.
\newblock {\em Proceedings of the Sydney Category Theory Seminar}, 420:2054,
  1974.

\bibitem{enriquemolinerheunentull:space}
P.~{Enrique Moliner}, C.~Heunen, and S.~Tull.
\newblock Space in monoidal categories.
\newblock In {\em Quantum Physics and Logic}, volume 266 of {\em Electronic
  Proceedings in Theoretical Computer Science}, pages 399--410, 2017.

\bibitem{fioreleinster:thompsonsgroupf}
M.~Fiore and T.~Leinster.
\newblock An abstract characterization of {T}hompson's group {$F$}.
\newblock {\em Semigroup Forum}, 80:325--340, 2010.

\bibitem{fiore2006pseudo}
T.~M. Fiore.
\newblock {\em Pseudo limits, biadjoints, and pseudo algebras: categorical
  foundations of conformal field theory}.
\newblock American Mathematical Society, 2006.

\bibitem{fritz:categoriesoffractions}
T.~Fritz.
\newblock Categories of fractions revisited.
\newblock {\em arXiv:0803.2587}, 2008.

\bibitem{fujiikatsumatamellies:gradedmonads}
S.~Fujii, S.~Katsumata, and P.-A. Melli{\`es}.
\newblock Towards a formal theory of graded monads.
\newblock In {\em Foundations of Software Science and Computation Structures},
  pages 513--530. Springer, 2015.

\bibitem{gabrielzisman:calculusoffractions}
P.~Gabriel and M.~Zisman.
\newblock {\em Calculus of fractions and homotopy theory}.
\newblock Springer, 1967.

\bibitem{grandis:cohesive}
M.~Grandis.
\newblock Cohesive categories and manifolds.
\newblock {\em Annali di Matematica pura ed applicata}, CLVII:199--244, 1990.

\bibitem{grillet1971regular}
P.~A. Grillet.
\newblock Regular categories.
\newblock In {\em Exact categories and categories of sheaves}, pages 121--222.
  Springer, 1971.

\bibitem{hasegawa2010bialgebras}
Masahito Hasegawa.
\newblock Bialgebras in rel.
\newblock {\em Electronic Notes in Theoretical Computer Science}, 265:337--350,
  2010.

\bibitem{heunen:embedding}
C.~Heunen.
\newblock An embedding theorem for {H}ilbert categories.
\newblock {\em Theory and Applications of Categories}, 22(13):321--344, 2009.

\bibitem{heunenkarvonen:monads}
C.~Heunen and M.~Karvonen.
\newblock Monads on dagger categories.
\newblock {\em Theory and Applications of Categories}, 31(35):1016--1043, 2016.

\bibitem{heunenreyes:frobenius}
C.~Heunen and M.~L. Reyes.
\newblock Frobenius structures over {H}ilbert {C}*-modules.
\newblock {\em Communications in Mathematical Physics}, 361(2):787--824, 2018.

\bibitem{hines:classicalstructures}
P.~Hines.
\newblock {\em Categories and Types in Logic, Language, and Physics}, volume
  8222 of {\em Lecture Notes in Computer Science}, chapter Classical structures
  based on unitaries, pages 188--210.
\newblock Springer, 2014.

\bibitem{hines:coherenceselfsimilarity}
P.~Hines.
\newblock Coherence and strictification for self-similarity.
\newblock {\em Journal of Homotopy and Related Structures}, 11:847--867, 2016.

\bibitem{hogancamp:idempotent}
M.~Hogancamp.
\newblock Idempotents in triangulated monoidal categories.
\newblock arXiv:1703.01001.

\bibitem{imkelly:day}
G.~B. Im and G.~M. Kelly.
\newblock A universal property for the convolution monoidal structure.
\newblock {\em Journal of Pure and Applied Algebra}, 43:75--88, 1986.

\bibitem{jacobsmandemaker:coreflections}
B.~Jacobs and J.~Mandemaker.
\newblock Coreflections in algebraic quantum logic.
\newblock {\em Foundations of Physics}, 42(2):932--958, 2012.

\bibitem{johnstone:stonespaces}
P.~T. Johnstone.
\newblock {\em Stone spaces}.
\newblock Cambridge University Press, 1982.

\bibitem{johnstone:elephant}
P.~T. Johnstone.
\newblock {\em Sketches of an elephant: A topos theory compendium}.
\newblock Oxford University Press, 2002.

\bibitem{joyal:chevalleytarski}
A.~Joyal.
\newblock Les theoremes de {C}hevalley-{T}arski et remarques sur l'algebre
  constructive.
\newblock {\em Cahiers de Topologie et G{\'e}ometrie Differentiele},
  XVI-3:256--258, 1975.

\bibitem{joyalstreetverity:traced}
A.~Joyal, R.~Street, and D.~Verity.
\newblock Traced monoidal categories.
\newblock {\em Mathematical Proceedings of the Cambridge Philosophical
  Society}, 119:447--468, 1996.

\bibitem{kaplansky:modules}
I.~Kaplansky.
\newblock Modules over operator algebras.
\newblock {\em American Journal of Mathematics}, 75:839--853, 1953.

\bibitem{kashiwara2005categories}
M.~Kashiwara and P.~Schapira.
\newblock {\em Categories and sheaves}.
\newblock Springer, 2005.

\bibitem{kasparov:hilbertmodules}
G.~G. Kasparov.
\newblock Hilbert {C}*-modules: {T}heorems of {S}tinespring and {V}oiculescu.
\newblock {\em Journal of Operator Theory}, 4:133--150, 1980.

\bibitem{kock:saavedra}
J.~Kock.
\newblock Elementary remarks on units in monoidal categories.
\newblock {\em Mathematical Proceedings of the Cambridge Philosophical
  Society}, 144:53--76, 2008.

\bibitem{kockpitsch:pointfree}
J.~Kock and W.~Pitsch.
\newblock Hochster duality in derived categories and point-free reconstruction
  of schemes.
\newblock {\em Transactions of the American Mathematical Society},
  369:223--261, 2017.

\bibitem{lance:hilbertmodules}
E.~C. Lance.
\newblock {\em Hilbert C*-modules: a toolkit for operator algebraists}.
\newblock Cambridge University Press, 1995.

\bibitem{maclane:categorieswork}
S.~{Mac Lane}.
\newblock {\em Categories for the Working Mathematician}.
\newblock Springer, 2nd edition, 1971.

\bibitem{menni:thesis}
M.~Menni.
\newblock {\em Exact completions and toposes}.
\newblock PhD thesis, University of Edinburgh, 2000.

\bibitem{quillen:nonunitalrings}
D.~Quillen.
\newblock Module theory over nonunital rings.
\newblock
  \url{http://www.claymath.org/library/Quillen/Working\_papers/quillen~1996/1996-2.pdf},
  1996.

\bibitem{resende:groupoidquantales}
P.~Resende.
\newblock {\'E}tale groupoids and their quantales.
\newblock {\em Advances in Mathematics}, 208:147--209, 2006.

\bibitem{rieffel:representations}
M.~A. Rieffel.
\newblock Induced representations of {C}*-algebras.
\newblock {\em Advances in Mathematics}, 13(2):176--257, 1974.

\bibitem{rosenthal:quantales}
K.~I. Rosenthal.
\newblock {\em Quantales and their applicatoins}.
\newblock Pitman Research Notes in Mathematics. Longman Scientific \&
  Technical, 1990.

\bibitem{rudin2006real}
W.~Rudin.
\newblock {\em Real and complex analysis}.
\newblock Tata McGraw-Hill Education, 2006.

\bibitem{selinger:graphicallanguages}
P.~Selinger.
\newblock A survey of graphical languages for monoidal categories.
\newblock In B.~Coecke, editor, {\em New Structures for Physics}, number 813 in
  Lecture Notes in Physics, pages 289--356. Springer, 2009.

\bibitem{street:formal}
R.~Street.
\newblock The formal theory of monads.
\newblock {\em Journal of Pure and Applied Algebra}, 2(2):149--168, 1972.

\bibitem{street:topos}
R.~Street.
\newblock Notions of topos.
\newblock {\em Bulletin of the Australian Mathematical Society}, 23:199--208,
  1981.

\bibitem{szlachanyi:skew}
K.~Szlachanyi.
\newblock Skew-monoidal categories and bialgebroids.
\newblock {\em Advances in Mathematics}, 231:1694--1730, 2012.

\bibitem{vickers:topology}
S.~Vickers.
\newblock {\em Topology via Logic}.
\newblock Cambridge University Press, 1989.

\end{thebibliography}

\appendix
\section{Day convolution}\label{sec:day}

This appendix describes in some detail the monoidal structure on presheaf categories given by Day convolution~\cite{day:monoidalmonads}, so that it can prove some of the lemmas of Section~\ref{sec:completion}. We start with the abstract definition, then give a concrete description, and use that to write down the coherence isomorphisms; we have no need for associators or the braiding in this article, so will not discuss these explicitly. Fix a monoidal category $\cat{C}$, and write $\Psh{\cat{C}}=\psh{C}$ for the category of presheaves. 

\subsubsection*{Tensor product of objects}

The Day convolution $F \ootimes G$ of presheaves $F,G \in \widehat{\cat{C}}$ is given abstractly as a left Kan extension
\[
  F \ootimes G \simeq \mathrm{Lan}_{\otimes}(F\times G)
\]
of the functor $F\times G\colon (\cat{C}\times\cat{C})\op \to \cat{Set}$, given by $(A,B) \mapsto F(A) \times G(B)$ and $(f,g) \mapsto F(f) \times G(g)$, along the tensor product $\otimes \colon (\cat{C} \times \cat{C})\op \to \cat{C}\op$ of the base category.
This left Kan extension may be computed~\cite[X.4.1]{maclane:categorieswork} as a coend
\[
  (F \ootimes G)(A) = \int^{B,C} \cat{C}(A,B \otimes C) \times F(B) \times G(C)\text.
\]
Now, coends can be computed as colimits~\cite[IX.5.1]{maclane:categorieswork}, and in turn, colimits can be constructed from coproducts and coequalizers~\cite[V.2.2]{maclane:categorieswork}. Thus $F \ootimes G$ is a coequalizer of the following two functions.
\begin{align*}
  \coprod_{\substack{f \colon B \to B' \\ g \colon C \to C'}} \cat{C}(A,B\otimes C) \times F(B') \times G(C')
  & \rightrightarrows
  \coprod_{B,C} \cat{C}(A,B \otimes C) \times F(B) \times G(C) \\
  (h,\,x,\,y)_{(f,g)} & \mapsto \big( (f \otimes g) \circ h,\, x,\, y\big)_{(B',C')} \\
  (h,\,x,\,y)_{(f,g)} & \mapsto \big( h,\, F(f),\, G(g) \big)_{(B,C)}
\end{align*}
Finally, coproducts in $\cat{Set}$ are disjoint unions, and coequalizers are quotients. Thus
\[
  (F \ootimes G)(A) = \Big(\coprod_{B,C} \cat{C}(A,B\otimes C) \times F(B) \times G(C)\Big) \slash \smash{\sim}\text,
\]
where $\sim$ is the least equivalence relation satisfying
\[
  (h,x,y)_{(B,C)} \sim (h',x',y')_{(B',C')}
\]
when there exist $f \colon B \to B'$ and $g \colon C \to C'$ such that $x=F(f)(x')$, $y=G(g)(y')$ and $(f \otimes g) \circ h=h'$.
\[\begin{tikzpicture}[xscale=2]
  \node (l) at (0,2.5) {$B\otimes C$};
  \node (r) at (2,2.5) {$B'\otimes C'$};
  \node (t) at (1,3.5) {$A$};
  \draw[->] (t) to node[above]{$h\text{ }$}(l);
  \draw[->] (l) to node[below]{$f\otimes g$} (r);
  \draw[->] (t) to node[above]{$\text{ }h'$} (r);
\end{tikzpicture}\]
It also follows that the action of $F \ootimes G$ on a morphism $f\colon A'\rightarrow A$ is given by $(h, \,x, \,y)_{(B,C)} \mapsto (h\circ f,\, x,\, y)_{(B,C)}$.

\subsubsection*{Tensor product of morphisms}

If $\varphi \colon F \Rightarrow F'$ and $\psi \colon G \Rightarrow G'$ are natural transformations, then so is $\varphi \ootimes \psi \colon F \ootimes G \Rightarrow F' \ootimes G'$, given by
\[
  (\varphi \ootimes \psi)_A \colon (h,\,x,\,y)_{(B,C)} \mapsto (h,\, \varphi_B (x),\, \psi_C (y))_{(B,C)}\text.
\]

\subsubsection*{Tensor unit}

If $I$ is the tensor unit of $\cat{C}$, then $\widehat{I} = \cat{C}(-,I)$ is the tensor unit of $\Psh{\cat{C}}$.

\subsubsection*{Unitors}

Write $\rho_A \colon A \otimes I \to A$ and $\lambda_A \colon I \otimes A \to A$ for the unitors in $\cat{C}$. The right unitor $\widehat{\rho}_F \colon F \ootimes \widehat{I} \Rightarrow F$ is given by
\[
  (\widehat{\rho}_F)_A \colon 
  (h,\, x,\, y)_{(B,C)} \mapsto F\big(\rho_B \circ (\id[B]\otimes y)\circ h\big)(x)\text.
\]
and the left unitor $\widehat{\lambda}_F \colon \widehat{I} \ootimes F \Rightarrow F$ by
\[
  (\widehat{\lambda}_F)_A \colon 
  (h,\,x,\,y)_{(B,C)} \mapsto F\big(\lambda_C \circ (x\otimes \id[C])\circ h\big)(y)\text.
\]
It is straightforward to check that these are well-defined natural isomorphisms.

\subsubsection*{Subunits}

A subunit $S$ is firstly a subobject of $\widehat{I}$, i.e.~a subfunctor of $\cat{C}(-,I)$. Equivalently, to each object $A$ it assigns a set $S(A)$ of morphisms $A \to I$, and naturality amounts to these being closed under precomposition with arbitrary morphisms of $\cat{C}$, i.e.~whenever $s \in S(A)$ and $f\colon B \to A$ then $s \circ f \in S(B)$. Finally $S$ being a subunit means precisely that for all $s \in S(A)$ there exists a unique $(h,x,y)_{(B,C)} \in (S \ootimes S)(A)$, for some $h \colon A \to B \otimes C$, $x \in S(B)$, and $y \in S(C)$, with $s = \rho_I \circ (x \otimes y) \circ h$.

\subsubsection*{Proof of {Proposition}~\ref{prop:PshFirmCounter}} 

By the above description, subunits in $\Psh{M}$ correspond to ideals $S\subseteq M$ which are idempotent in the sense that $S=SS$, and furthermore satisfy the requirement that the map $S \ootimes S \to S$ is injective. 

Let $S$ be the ideal consisting of all elements of the form $(a,0) + x$ for some $a >0$, and $T$ the ideal of all elements of the form $(0,b) + y$ for $b > 0$, similarly. We claim that these are subunits. If $\Psh{M}$ were firm, then $S \ootimes T = S \cap T$ being a subunit and hence idempotent as an ideal. But $S \cap T$ is not idempotent. 

Indeed, consider $(0,1) \in S \cap T$. Now suppose that $(0,1) = (a,b) + (c,d)$ for some $(a,b),(c,d) \in S \cap T$. Then $a + b + c + d = 1$. If $a+c < 1$ necessarily $a=c=0$. Now $b>0$ or $d>0$, so either $b<1$ or $d<1$; without loss of generality say $b<1$. But this contradicts $(a,b) \in S$. Therefore $a+c=1$. But then $b=d=0$, contradicting $(a,b) \in T$. Thus $S \cap T$ is not idempotent.

It remains to verify that $S$ and $T$ are subunits. We first treat the case for $S$.
 Firstly, $S$ is idempotent since each element $(a,0)$ for $a > 0$ has $(a,0) = (a/2,0) + (a/2,0)$ with $(a/2,0) \in S$.
Finally, we must check that any $(h,s,t) \in S \ootimes S$ is determined by its value $hst \in M$.

Note that $(h,s,t)\sim (h + x + y,s',t')$ when $s=s'+ x$ and $t=t' + y$ for $h,x,y \in M$ and $s,s',t,t' \in S$. Hence any element $(h,s,t)$ is equivalent to one of the form $\big((b,c),(a,0),(a,0)\big)$ for arbitrarily small $a > 0$. Now suppose that 
\[
  (b,c) + (a,0) + (a,0) = (b',c') + (a',0) + (a',0)
\]
Using the same trick again we may assume that $a'=a$. Now if $b + a + a > 1$ there is some $d < a$, say with $a = d + e$, such that $b + d + d > 1$ also. Letting $(b',c')=(b,c)+(2d,0)$ gives $\big((b,c),(a,0),(a,0)\big)=\big((b',c'),(e,0),(e,0)\big)$, now with $b' + e + e < 1$. Applying this trick we may have assumed to begin with that $b + a + a < 1$ and $b' + a + a < 1$. But this ensures that $b=b'$ and $c=c'$, and we are done. Seeing that $T$ is a subunit is similar but simpler.
\qed

\subsubsection*{Proof of Lemma~\ref{lem:broadtensor}} 

As noted when proving Proposition~\ref{prop:broadidempotent}, we may assume that $U$ and $V$ are idempotent. The lower isomorphism of~\eqref{eq:day-useful} follows from monoidality of the Yoneda embedding. 

By definition, $\big(\broad{U}{X} \ootimes \broad{V}{Y}\big)(A)$ consists of triples $(h,f,g)$ where $h \colon A \to B \otimes C$, $f \colon B \to X$ restricts to $U$, and $g \colon C \to Y$ restricts to $V$, subject to the Day identification rules. From the definition of the monoidal structure in $\Psh{\cat{C}}$, the transformation $u \ootimes v$ in~\eqref{eq:day-useful} has component at $A$ given by
\[
  (h, f, g) \mapsto ( (f \otimes g) \circ h \colon A \to X \otimes Y )
\]
Since this is well-defined and each such morphism $(f \otimes g) \circ h$ clearly restricts to a member of $U \otimes V$, it restricts to a transformation as in the top row of~\eqref{eq:day-useful}, making the diagram commute. Furthermore each such map is surjective since any morphism $k \colon A \to X \otimes Y$ restricting to a member of $U \otimes V$ has, using the braiding, that $k = ((s \otimes X) \otimes (t \otimes Y)) \circ h $ for some $h$, $s \in U$ and $t \in V$ so that $(h, u \otimes X, v \otimes Y) \mapsto k$. 

Finally, we show injectivity. For any triple $(h,f,g)$ with $f = (s \otimes X) \circ \bar{f}$ and $g = (t \otimes Y) \circ \bar{g}$ for some $\bar{f}, \bar{g}, s \in U$ and $t \in V$,
\[
  (h,f,g) \sim ((\bar{f} \otimes \bar{g}) \circ h, s \otimes X, t \otimes Y)
\]
by the Day identification rules, and so it suffices to consider triples of this form. Now if $(h, s \otimes X, t \otimes Y)$ and $(h', s' \otimes X, t' \otimes Y)$ are mapped to the same morphism then it restricts to $s \wedge s' \in U$ and $t \wedge t' \in V$, so that for some $k$:
\begin{align*}
  (k, (s \wedge s') \otimes X, (t \wedge t') \otimes Y)
  & \sim (h, s \otimes X, t \otimes Y) \\ 
  & \sim (h', s' \otimes X, t' \otimes Y) 
\end{align*}
by definition of~$\sim$, making these triples equivalent as required.
\qed

\end{document}